\documentclass[a4paper, 11pt, english]{article}
\usepackage[utf8]{inputenc}
\usepackage[T1]{fontenc}
\usepackage{babel}
\usepackage{graphicx}
\usepackage{stmaryrd}
\usepackage[a4paper]{geometry}
\usepackage{amsmath,amsfonts,amssymb,amsthm,epsfig,epstopdf,url,array}
\usepackage{rotating}
\usepackage[colorlinks=true,citecolor=red,linkcolor=blue,pdfpagetransition=Blinds]{hyperref}
\usepackage{cleveref}
\usepackage{nameref}
\usepackage{enumitem}
\usepackage{comment}
\Crefname{paragraph}{Section}{Sections}
\usepackage{fancyhdr}

\usepackage{ulem}

\usepackage{fullpage}

\newcommand{\ensemblenombre}[1]{\mathbb{#1}}

\newcommand{\R}{} % Probleme LateXML
\renewcommand{\R}{\ensemblenombre{R}}

\renewcommand{\le}{\leqslant}
\renewcommand{\ge}{\geqslant}

\newcommand{\dive}[1]{\mathrm{div}}

\newcommand{\eps}{\varepsilon}

\newcommand{\ts}{\textstyle}
\newcommand{\ds}{\displaystyle}

\theoremstyle{plain} 
\newtheorem{pr}{Proposition}[section] 
\newtheorem{tm}{Theorem}
\newtheorem{thm}{Theorem}[section] 
\newtheorem{lm}[pr]{Lemma}
\newtheorem{cor}[tm]{Corollary}
\theoremstyle{definition}
\newtheorem{df}[pr]{Definition}
\newtheorem{rmk}[pr]{Remark}

\numberwithin{equation}{section}

\newtheorem{atheorem}{Theorem} 

\newcommand{\be}{\begin{equation} \label}
\newcommand{\ee}{\end{equation}}

\makeatletter
\let\original@addcontentsline\addcontentsline
\newcommand{\dummy@addcontentsline}[3]{}
\newcommand{\DeactivateToc}{\let\addcontentsline\dummy@addcontentsline}
\newcommand{\ActivateToc}{\let\addcontentsline\original@addcontentsline}
\makeatother

\begin{document}

\title{Blow-up for weakly superlinear heat equations and blow-up controllability of the linear heat equation}

\author{Kévin Le Balc'h, Philippe Souplet}

\maketitle

\begin{abstract}
The aim of this article is twofold: 
(a) to revisit the blow-up theory of weakly superlinear heat equations;
(b) to explore the notion of internal global/regional blow-up controllability for the linear heat equation.

Regarding point (a), we consider nonnegative nonlinearities which grow 
like \( u \log^p{\hskip -1.5pt}|u| \) for large $|u|$ with \( p \in (1,2] \) and may be spatially localized.
For \( p \in (1,2) \) (respectively, \( p = 2 \)), with nonnegative initial data and Dirichlet boundary conditions, we prove that if the existence time is finite, then the blow-up set is global (respectively, at least regional),
and provide the precise upper and lower blow-up estimates.
These results are entirely new in cases where the nonlinearity is spatially localized,
and they significantly improve on known results even in the non localized case.  The proofs combine interpolation 
and comparison arguments, 
test-function methods, suitable smoothing effects, and heat kernel estimates. The continuity property of the existence time
with respect to parameter and initial data is also established.

As for point (b), as an application of (a), we establish that the linear heat equation (with Dirichlet boundary conditions) is small-time globally blow-up controllable. This means that for any open subset $\omega$, any time \( T > 0 \) and any initial data, there exists a control law spatially localized in $\omega$ such that the corresponding controlled solution of the heat equation undergoes global blow-up at time \( t = T \). Additionally we obtain regional blow-up controllability results. These new controllability results bypass the traditional use of exact-controllability results to well-prepared initial datum. The control is chosen as a “feedback law” where the feedback depends both on the horizon time $T$, the initial data $u_0$ and the state $u(t)$. The proof uses precise blow-up properties of localized weakly superlinear heat equations, coming from point (a). 
\end{abstract}

%\tableofcontents

\section{Introduction}

The study of parabolic partial differential equations (PDEs) set in a bounded domain $\Omega$, with appropriate boundary conditions, has been a cornerstone of analysis, particularly due to their role in modelling diffusion processes. The heat equation, a fundamental example of a parabolic PDE, has been extensively studied for its properties of existence, uniqueness, and regularity of solutions. Early research primarily focused on understanding these aspects, but it soon became evident that nonlinearities could lead to solutions becoming unbounded in a finite time $T>0$, a phenomenon known as “blow-up”. Starting from the pioneering works of Kaplan \cite{Kap63} and Fujita \cite{Fuj66}, there was a development of various techniques to analyze blow-up phenomena for the problem
\begin{equation}
	\label{eqfu0}
		\left\{
			\begin{array}{ll}
				 \partial_t u - \Delta u = f(u)& \text{ in }  (0,T) \times \Omega, 
				\\
				u = 0 & \text{ on } (0,T)\times \partial \Omega, 
				\\
				u(0, \cdot) = u_0 & \text{ in } \Omega.
			\end{array}
		\right.
\end{equation}
Significant contributions by many authors were made to identify conditions on the nonlinearity $f$ and on the data $u_0$ under which blow-up occurs, determine the blow-up set and characterize the blow-up rates and profiles 
(see the monograph \cite{QS19} for a detailed account). In particular, it was shown in the classical works \cite{We84, FML85, MW85} that single point blow-up may occur for power nonlinearities $f(u)=|u|^{p-1}u$ with $p>1$ or for the exponential nonlinearity $f(u)=e^u$.
For such $f$, if for instance $\Omega=B _R$ is a ball and $0\le u_0\in L^\infty(\Omega)$ is radially symmetric nonincreasing,
then blow-up can occur only at the origin.

The possibility of ``large'' blow-up sets for reaction-diffusion equations of the form \eqref{eqfu0} was also investigated, leading to the study of \eqref{eqfu0} for weakly superlinear nonlinearities $f(u)=u\log^p(2+|u|)$.
For $p\in(1,2)$
and suitable $0\le u_0\in L^\infty(\Omega)$, it was shown in \cite{Lac86} that global blow-up occurs, namely $\lim_{t\to T} u(t,x)=\infty$ for all $x\in\Omega$ (see \Cref{rm:blowupremark} below for more details).
Interestingly, $p=2$ was shown to be a critical exponent: if $\Omega=B _R$ is a ball and $0\le u_0\in L^\infty(\Omega)$ is radially symmetric nonincreasing, then blow-up can occur only at the origin if $p>2$ (see \cite[Section~2]{FML85}), whereas regional blow-up may occur for $p=2$ 
(see \cite{Lac86}).
Problem \eqref{eqfu0} with weakly superlinear nonlinearities was further investigated in \cite{GV93,GV96}, 
where detailed blow-up asymptotics were obtained in case of radially symmetric solutions in a ball or in $\R^N$.

Global or regional blow-up has also been observed and studied for other classes of parabolic equations.
We refer to, e.g.,~\cite{AV95,TI97,Win03,Co11,AI17,AIU25}
for degenerate equations arising in the study of curve shortening flows, \cite{BBL88, De95, Sou99, KLW17}
 for equations with nonlocal sources, 
and \cite{GS04,FPR07,FPL07,YS20} for problems with boundary sources.

Our first line of results goes in the direction of the works \cite{Lac86,GV93,GV96},
by considering problem \eqref{eqfu0} with nonlinearities $f=f(x,u)$ depending on the spatial variable~$x$,
with a double motivation: (a) understanding the conjunction of two features of the nonlinearity: 
a weakly superlinear growth and a possible spatial localization of the support,
with the aim of determining the blow-up set and obtaining blow-up rates; (b) refining the description of blow-up asymptotics 
in nonradial situations, even in the case $f=f(u)$.
\\

On the other hand, the controllability of parabolic equations, initiated in the seminal work of Fattorini and Russell \cite{FR71}, has been a central topic in control theory over the past fifty years. Broadly speaking, controllability concerns the ability to steer the solution of a parabolic PDE by means of external inputs. For the internal control of the linear heat equation, a fundamental milestone was established independently by Lebeau–Robbiano \cite{LR95} and Fursikov–Imanuvilov \cite{FI96}, who proved the small-time null-controllability of the heat equation: for any open subset $\omega\subset\Omega$, any $T>0$ and any initial datum, one can find a control supported in~$(0,T)\times\omega$ so that the controlled solution vanishes identically at time $t=T$.
Since then, much effort has been devoted to controllability of more complex models: for instance, linear coupled parabolic systems (see the survey \cite{AKBGBdT11}) and nonlinear parabolic equations. In the latter direction, Fernández-Cara and Zuazua \cite{FCZ00} proved small-time global null-controllability for weakly superlinear heat equations; in particular, the localized control can prevent blow-up, see also \cite{Bar00, LB20}.
A different perspective was initiated in \cite{Lin18}, where the notion of blow-up controllability was introduced: it was shown that the heat equation is small-time blow-up controllable, meaning that one can drive the solution to blow up at a prescribed time by means of a localized control. Subsequently, \cite{LZ22} refined the result by prescribing the blow-up set to a single point inside the control region. Our second main results are in the same direction. We investigate whether a localized control can force the solution of the linear heat equation to undergo regional or global blow-up.\\
 
%We first begin by introducing some notation.
Throughout this paper, $\Omega$ is a bounded $C^2$ domain of $\R^N$, $N \geq 1$,
and we denote the distance to the boundary by $\delta(x)=\delta(x,\partial\Omega)$. 
%The notation $\omega$ typically stands for a nonempty open subset of $\Omega$.
Also we shall denote $\Omega_\eps=\{x\in\Omega\ ;\, \delta(x)>\eps\}$.

\subsection{Blow-up for semilinear heat equations with a localized nonlinearity}
\label{secBU}

We aim to study semilinear initial boundary value problems of the form
\begin{equation}
	\label{eqfu}
		\left\{
			\begin{array}{ll}
				 \partial_t u - \Delta u = f(x,u)& \text{ in }  (0,T) \times \Omega,  
				\\
				u = 0 & \text{ on } (0,T)\times \partial \Omega, 
				\\
				u(0, \cdot) = u_0 & \text{ in } \Omega.
			\end{array}
		\right.
\end{equation}
%We assume that
%\begin{equation}
%\label{eq:hypothesisf0}
%\hbox{$f=f(x,s)$ is locally bounded on $\overline\Omega\times\R$ and $C^1$ with respect to $s$},
%\end{equation}
Assume
\begin{equation}
\label{eq:positivedata}
0\le u_0\in L^\infty(\Omega).
\end{equation}
If $f$ is locally bounded on $\overline\Omega\times\R$ and $C^1$ with respect to $s$ then, by standard theory
(see for instance \cite[Definition 15.1 and Proposition 51.40]{QS19}), problem \eqref{eqfu} admits a unique maximal solution $u\in L^\infty_{loc}([0,T);$ $L^\infty(\Omega))$,
where $T=T(u_0)\in(0,\infty]$ denotes the maximal existence time. 
If also $f(\cdot,0)\ge 0$, then $u\ge 0$. Moreover, if $T<\infty$, then $u$ blows up in the $L^\infty$ sense, 
i.e.~$\lim_{t\to T} \|u(t)\|_\infty=\infty$. Classically, according for instance to \cite[Theorems 17.1, 17.3]{QS19}, blow-up can indeed occur for the semilinearity $f=f(s) = s^p$, with $p>1$ or even for the weak semilinearity $f=f(s) = s \log^{p}(1+s)$, with $p>1$ provided
$ \int_{\Omega} u_0 \Phi_1 dx \geq C(\Omega,f),$
where $\Phi_1$ is the first Dirichlet Laplace eigenfunction in $\Omega$.

 In this paper, we investigate the case of weak
 localized semilinearities, i.e.~
\begin{equation}
\label{eq:localizedf}
\text{supp}\ f(\cdot,s) \subset \overline{\omega},\qquad s \in \R,
\end{equation}
whose typical example  is
\begin{equation}
\label{eq:modelf}
f(x,s) = K s\log^p(a+|s|)1_\omega(x),\qquad (x,s) \in \Omega \times \R, 
\end{equation} 
with $p>1$, $K>0$, $a\ge 1$.
Here we will assume that
\be{eq:hypothesisomega}
\hbox{$\omega\subset\Omega$ is a $C^2$ smooth domain.}
\ee
We note that $T<\infty$ whenever, for instance,
$\int_{\omega} u_0 \varphi_1 dx \geq C(\omega,p,K,a),$
where $\varphi_1$ is the first Dirichlet Laplace eigenfunction in $\omega$,
as a consequence of the comparison principle and of the above criterion.

For characterizing the possible locations of the blow-up, we introduce the following standard definition. 

\begin{df}
For $0 \leq u_0 \in L^{\infty}(\Omega)$ and $T=T(u_0) < \infty$,  the blow-up set is defined as
\begin{equation}
\label{eq:defblowup}
B(u_0) := \big\{x \in \overline{\Omega}\ ;\ \exists (x_k, t_k) \in \Omega \times (0,T)\ \text{such that}\ (x_k,t_k) \to (t,x),\ |u(x_k,t_k)| \to +\infty\big\}.
\end{equation}
We note that $B(u_0)$ is a closed subset of $\overline{\Omega}$. 
Blow-up 
is said to be \textit{single-point} if $B(u_0) = \{x_0\}$ for some $x_0 \in \overline{\Omega}$, \textit{regional} if $B(u_0)$ contains a nonempty open subset of $\Omega$ and $B(u_0) \neq \overline{\Omega}$, and \textit{global} if $B(u_0) = \overline{\Omega}$.
\end{df}
Throughout this paper, $c_i(\Omega), C_i(\Omega)$ denote various positive constants depending only on $\Omega$.

Our first main result concerning \eqref{eqfu}, \eqref{eq:modelf}  shows \textit{global blow-up} 
in the range $1<p<2$ and provides a precise two-sided blow-up 
estimate.

\begin{tm}
\label{cor:globalblowup0BIS}
Let $1<p<2$, $K>0$, $a\ge 1$, assume \eqref{eq:hypothesisomega} and consider problem \eqref{eqfu} where $f$ is given by 
\eqref{eq:modelf} and $u_0$ satisfies \eqref{eq:positivedata} and $T=T(u_0)<\infty$.
Then $u$ undergoes global blow-up and satisfies the two-sided estimate:
\begin{equation}
\label{twosidedBU}
c\exp\bigl[c(T-t)^{-1/(p-1)}\bigr]\le \frac{u(t,x)}{\delta(x)}\le C\exp\bigl[C(T-t)^{-1/(p-1)}\bigr],\quad T/2<t<T,\ x\in\Omega,
\end{equation}
for some constants $c, C>0$.
 \end{tm}

 In the critical case $p=2$, we obtain the following conclusions on \textit{regional and global blow-up},
along with lower and upper blow-up estimates and information on the blow-up set.

 \goodbreak
 
\begin{tm}
\label{cor:globalblowup0TER}
Let $p=2$, $K>0$, $a\ge 1$, assume \eqref{eq:hypothesisomega} and consider problem \eqref{eqfu} where $f$ is given by 
\eqref{eq:modelf} and $u_0$ satisfies \eqref{eq:positivedata} and $T=T(u_0)<\infty$.
%Under the assumptions of Theorem~\ref{cor:globalblowup0BIS} with $p=2$, we have the following conclusions.
\begin{itemize} 
\item[(i)] 
Then $u$ undergoes at least regional blow-up. 
Namely, there exists an open ball $B\subset\subset\omega$ such that
$\ds\lim_{t\to T} \bigl(\min_{x\in\overline B}u(t,x)\bigr)=\infty$.
Moreover $u$ satisfies the upper blow-up estimate:
\be{uppernew00}
u(t,x)\le C\exp\bigl[C(T-t)^{-1}\bigr]\delta(x),\quad (t,x) \in (T/2,T)\times\Omega,
\ee
for some constant $C>0$.

\item[(ii)]The blowup set satisfies
	$B(u_0)\subset \big\{x\ ;\, {\rm dist}(x,\omega)\le C_0K^{-1/2}\big\}$, where 
$C_0>0$ depends only~$N$.
In particular, if $\omega\ne \Omega$ and $K>C_0^2\, (\sup_{x \in \Omega} {\rm dist}(x, \omega))^{-2}$,
%% \Omega \setminus \omega ->  \Omega
then blow-up is only regional.

\item[(iii)]  Assume $K\le c_1(\Omega) \,diam^{-2}(\omega)$. Then 
$\overline\omega\subset B(u_0)$ and 
	$$\min_{x\in\overline\omega} u(t,x)\ge c\exp\bigl[c(T-t)^{-1}\bigr],\quad T/2<t<T,$$
	for some constant $c>0$.
	If moreover, $\omega\ne \Omega$ (resp.,~$\omega\subset\subset \Omega$), 
	then $\overline\omega\subsetneq B(u_0)$ (resp.,~$\bar\omega\subset\subset B(u_0)$).

\item[(iv)]If $K\le c_2(\Omega)$, then $u$ undergoes global blow-up
and satisfies, 	for some constant $c>0$,
\begin{equation}
\label{eq:lowerbupKsmallp2}
u(t,x)\ge\strut c\exp\bigl[c(T-t)^{-1}\bigr]\delta(x),\quad T/2<t<T,\ x\in\Omega,
%%c\exp\bigl[c(T-t)^{-1}\bigr]\le \frac{u(t,x)}{\delta(x)},\quad T/2<t<T,\ x\in\Omega.
\end{equation}
\end{itemize}
 \end{tm}

As a direct consequence of Theorem~\ref{cor:globalblowup0TER}(ii) and (iii),
for well adjusted constant $K$ and $\omega$ a ball,
we obtain the following qualitatively precise localization of the regional blow-up set
and two-sided estimate.

\begin{cor}
\label{cor:reg}
Under the assumptions of Theorem~\ref{cor:globalblowup0TER},
let $r>0$, $\kappa \in(0,\tfrac14 c_1(\Omega)]$, $K = \kappa r^{-2}$, $\omega=B_r$ and
	assume $\bar B_{\tilde \kappa r} \subset \Omega$ with $\tilde\kappa = 1 + C_0 \kappa^{-1/2}$.
Then 
\begin{equation}
\label{eq:buregionaltwoballs}
\overline{B_r} \subset B(u_0) \subset B_{\tilde \kappa r}
\end{equation}
and, for some constants $c,C>0$,
\begin{equation}
\label{eq:twoboundsregional}
c\exp\bigl[c(T-t)^{-1}\bigr] \le u(t,x)\le C\exp\bigl[C(T-t)^{-1}\bigr],\quad T/2<t<T,
\ x\in\overline{B_r}.
\end{equation}
\end{cor}

We note that, under the assumptions of Corollary~\ref{cor:reg}, the existence time $T$ is finite whenever 
$a\ge 2$, $\kappa\log^p a\ge C(N)$ and $u_0 1_{B_r}\not\equiv 0$; see Remark~\ref{rem-lmeigenf}.
We complete the above results by the following continuity property of the existence time
with respect to parameter, subdomain and initial data.
Beside its intrinstic interest it will be an important ingredient in the proof of our controllability results.

\begin{tm}
\label{tm:cbt}
Let $p>1$, $K\ge 0$, $a\ge 1$, consider problem \eqref{eqfu} where $f$ is given by 
\eqref{eq:modelf} and $u_0$ satisfies \eqref{eq:positivedata}.

\begin{itemize} 
\item[(i)]  Then the maximal existence time function $T=T(K,u_0)$ is continuous from $[0,\infty)\times L^\infty(\Omega)$ to $(0,\infty]$.

\item[(ii)]  Assume $x_0\in\Omega$ and $\omega=B(x_0,r)$. 
Then the function $T=T(K,u_0,r)$ is continuous from $[0,\infty)\times L^\infty(\Omega)\times(0,\delta(x_0))$ to $(0,\infty]$.
\end{itemize} 
\end{tm}

The following remarks are in order.

\begin{rmk}
\label{rm:blowupremark}
(a) Theorems~\ref{cor:globalblowup0BIS} and \ref{cor:globalblowup0TER} are completely new in the case $\omega\ne\Omega$
and partly new even in the case $\omega=\Omega$. 
Indeed, the results in \cite{Lac86} actually require the missing, additional assumption $u_t\ge 0$,
as they make use of a Bernstein-type gradient estimate from \cite[Theorem~3.1]{FML85}, where this assumption is also missing 
(see \cite[Proposition~24.4a]{QS19} and \cite{QuiSou} for details). 
Also the results in \cite{Lac86} provide no upper (but only lower) blow-up estimates,
whereas those in \cite{GV93,GV96} provide sharp two-sided estimates but only in radially symmetric situations
(making use of typically 1$d$ techniques of zero-number and comparison with special explicit solutions).
Here, owing to some new arguments (which in particular completely avoid the use of gradient estimates; see Remark~\ref{ideaspf} for details),
we can treat the case of localized nonlinearities and improve the results from \cite{Lac86}
for general $\Omega$,
including precise two-sided blow-up estimates, without needing the extra assumption $u_t\ge 0$.

\smallskip

(b) Under the assumptions of Theorem~\ref{cor:globalblowup0BIS} with $p>2$, we have $B(u_0) \subset \bar\omega$ 
(see Theorem~\ref{pr:globalblowup0BIS3-0}(iii) below), which departs from Theorem~\ref{cor:globalblowup0TER}(iv) and confirms the critical role of the value $p=2$.
Recall also \cite{FML85} that only single-point blow-up occurs when $p>2$,
$\Omega=\omega=B _R$ is a ball and $0\le u_0\in L^\infty(\Omega)$ is radially symmetric nonincreasing.

\smallskip

(c) For $p=2$ and $K>0$, under the assumptions of Theorem~\ref{cor:globalblowup0TER}(i),
we also have the lower blow-up estimate:
\begin{equation}
\label{lowerBUreg}
\min_{x\in \bar B_\rho({y_0}(t))} u(t,x)\ge c\exp\bigl[c(T-t)^{-1}\bigr], 
\quad t_0<t<T,
\end{equation}
for some numbers $c,\rho>0$, {$t_0\in(0,T)$} and some function $y_0: [t_0,T)\to \omega$
(see Remark~\ref{Remy0}).
However we do not know in general if \eqref{lowerBUreg} holds in a fixed ball.
In other words, we cannot in general rule out the possibility that the quasi-maximum points of $u(\cdot,t)$ 
oscillate in time between two or more separated regions.
By Theorem~\ref{cor:globalblowup0TER}(iv), the existence of such fixed ball is true if $K>0$ is suitably small.
Also, for any $K>0$, the existence of such fixed ball is true at least in the special situation
when $\Omega=B_R$, $\omega=B_r$, with $0<r\le R<\infty$
and $u_0$ radially symmetric and nonincreasing in~$|x|$.
Indeed, in this case, \eqref{uppernew00} and \eqref{lowerBUreg}  imply
\be{eq:twosideestimate2}
c\exp\bigl[c(T-t)^{-1}\bigr]\le u(t,x)\le C\exp\bigl[C(T-t)^{-1}\bigr],\ \ t_0<t<T,\ 0\le |x|\le \rho
\ee
(owing to $\min_{\bar B_\rho} u(t,\cdot)=u(t,\rho)\ge \min_{\bar B_\rho(y_0(t))} u(t,\cdot)$,
since $u(\cdot,t)$ is then radial decreasing).

\smallskip

(d) For the nonlinearity $f(u)=|u|^{p-1}u$ with subcritical power $p\in(1,p_S)$, where $p_S$
is the Sobolev exponent,  the continuity of the blow-up time with respect to initial data
was proved in \cite{BC87, Me92} (positive solutions) and in \cite{Q03} (general solutions).
Such results are far from trivial and the subcritical assumption is not technical: this property fails whenever $p>p_S$ (see \cite[Section~22.4]{QS19}).
This continuity property plays a significant role in the proof of several important results, such as complete blow-up \cite{BC87}
or the construction of solutions with prescribed blow-up points \cite{Me92}.
The continuity of the blow-up time for more general nonlinearities satisfying $cu^q\le f(x,u)\le Cu^p$ as $u\to\infty$
with $1<q<p<p_S$ (plus additional assumptions) are treated in \cite{BC87,Q03},
and further H\"older continuity properties can be found in \cite{GRZ}.
Theorem~\ref{tm:cbt} seems to be the first result of this type for weakly superlinear nonlinearities.
\end{rmk}

\begin{rmk} \label{ideaspf}
Theorems~\ref{cor:globalblowup0BIS}, \ref{cor:globalblowup0TER} and \ref{tm:cbt} are consequences of Theorems~\ref{cor:globalblowup0BIS2} 
and \ref{pr:globalblowup0BIS3-0} below,
obtained for more general nonlinearities satisfying lower and/or upper growth assumptions,
and where the precise dependence of the constants in the lower and upper blow-up estimates
with respect to parameters and initial data are given.

\medskip

The proofs of the lower bounds in
Theorems~\ref{cor:globalblowup0BIS} and \ref{cor:globalblowup0TER} are based on the following arguments
(cf.~Steps 1-5 of the proof of Theorem~\ref{cor:globalblowup0BIS2}):
\begin{itemize}[itemsep=-1mm]
\item a lower blow-up estimate of $\|u(t)\|_{L^{\infty}(\Omega)}$ by a comparison argument to the ODE,
\item a lower estimate $\|f(x,u)\|_{L^1(0,t;L_{\delta}^1(\Omega))} \geq \|u(t)\|_{L_{\delta}^1(\Omega)} - C$ relying on the 
eigenfunction method  (where $L_{\delta}^1$ is the $L^1$ space weighted by the distance to the boundary),
\item a lower blow-up estimate of $\|f(x,u)\|_{L^1(0,t;L_{\delta}^1(\Omega))}$ by $L_{\delta}^{1}$-$L^{\infty}$ smoothing effects,
\item a lower pointwise blow-up estimate by using 
lower bounds of the heat kernel and covering arguments.
\end{itemize}

We note that, whereas the first and fourth items were already used in \cite{Lac86}, the second and third ones are new with respect to \cite{Lac86}
(and enable one to avoid the use of Bernstein-type gradient estimates,
which, in turn, would not be suitable for localized nonlinearities).

\medskip

The proofs of the upper bounds in
Theorems~\ref{cor:globalblowup0BIS} and \ref{cor:globalblowup0TER} rely on:
\begin{itemize}[itemsep=-1mm]
\item an upper $L^1$ blow-up estimate of $u(\cdot,t)$ in $\omega$
obtained by test-function arguments, especially making use of the square of the first eigenfunction of the Dirichlet Laplacian on $\omega$ (see Lemma~\ref{lmeigenf}),
\item an upper blow-up estimate of $\|u(t)\|_{L^{\infty}(\Omega)}$ using an $L^{1}$-$L^{\infty}$ smoothing effect
based on interpolation and bootstrap arguments, along with 
bounds of the heat kernel (see the proof of Theorem~\ref{pr:globalblowup0BIS3-0}).
\end{itemize}

As for Theorem~\ref{tm:cbt}, it is based on a contradiction argument based on an
$L^\infty$ a priori estimate (valid for blow-up as well as global solutions -- see Theorem~\ref{pr:globalblowup0BIS3-0}(i)).
\end{rmk}

\subsection{Internal blow-up controllability of the linear heat equation}
 
 We consider the internal control of the linear heat equation
\begin{equation}
	\label{eq:HeatControl}
		\left\{
			\begin{array}{ll}
				 \partial_t u - \Delta u = h 1_{\omega} & \text{ in }  (0,+\infty) \times \Omega, 
				\\
				u = 0 & \text{ on } (0,+\infty)\times \partial \Omega, 
				\\
				u(0, \cdot) = u_0 & \text{ in } \Omega.
			\end{array}
		\right.
\end{equation}
In \eqref{eq:HeatControl}, at time $t \in [0,+\infty)$, $u(t, \cdot) : \Omega \to \R$ is the\textit{ state }while $h(t, \cdot) : \omega \to \R$ is the \textit{control}, i.e.~$h$ is a variable that one can choose in the system to act on the dynamics of $u$. Roughly speaking, by fixing a time $T>0$, an initial datum $u_0=u_0(x)$ and final target $u_f=u_f(x)$, we say that \eqref{eq:HeatControl} is \textit{controllable} from $u_0$ to $u_f$ in time $T$ if there exists a control $h = h(t,x)$ such that the solution $u=u(t,x)$ of \eqref{eq:HeatControl} satisfies $u(T) = u_f$. This is the so-called classical notion of \textit{controllability}. 
In view of 
the regularizing effects of the parabolic equation \eqref{eq:HeatControl}, that prevents from driving the solution $u$ to every final state $u_f$ in $L^2(\Omega)$ for instance, the good notion of controllability for \eqref{eq:HeatControl} is the \textit{null-controllability} i.e.~$u_f=0$. 
From the seminal papers \cite{LR95} and \cite{FI96}, the following result holds.

\renewcommand{\theatheorem}{\Alph{atheorem}}

\begin{atheorem}[\cite{LR95,FI96}]
\label{tm:nullcontrollable} 
{\it Let $\omega \subset \Omega$ be a nonempty open subset of $\Omega$. The linear heat equation \eqref{eq:HeatControl} is small-time null-controllable, i.e.~for every $T>0$, $u_0 \in L^{\infty}(\Omega)$, there exists $h \in L^{\infty}((0,T)\times\omega)$ such that the solution $u$ of \eqref{eq:HeatControl} satisfies $u(T,\cdot) = 0$.}
\end{atheorem}

From a modelling perspective, \Cref{tm:nullcontrollable} implies that in a room $\Omega$, one can drive the temperature to zero throughout the entire room within any arbitrarily short time using a localized heater or cooler in $\omega$. This result leads to a development of an important direction of research over the last thirty years, see for instance the survey \cite{AKBGBdT11}. Another natural question concerning \eqref{eq:HeatControl} is the characterization of the data $u(T,\cdot)$ that can be reached, starting from $u_0 \in L^{\infty}(\Omega)$, by acting locally through a control $h \in L^{\infty}((0,T)\times\omega)$. The identification of the so-called reachable space is not fully understood yet, even if important progress have been made recently in the case of a boundary control, see for instance \cite{Tuc23} for a description of such results.\\
  
In this part, motivated  
by recent works from \cite{Lin18} and \cite{LZ22}, we investigate a quite different question of controllability, namely \textit{blow-up controllability}. First, from \cite{Lin18}, we have the following result.

\begin{atheorem}[\cite{Lin18}]
\label{tm:lin}
{\it Let $\omega \subset \Omega$ be a nonempty open subset of $\Omega$. Then the linear heat equation \eqref{eq:HeatControl} is small-time blow-up controllable, i.e.~for every $T>0$, $u_0 \in L^{\infty}(\Omega)$, there exists $h \in L^{\infty}_{loc}([0,T);L^{\infty}(\omega))$ such that the solution $u$ of \eqref{eq:HeatControl} blows up at time $t=T$.}
\end{atheorem}

Then \cite{LZ22} refined this result by proving that they can actually prescribe the blow-up set of the solution to a 
unique point $x_0 \in \omega$. 

\begin{atheorem}
[\cite{LZ22}]
\label{tm:linzaag}
{\it Let $\omega \subset \Omega$ be a nonempty open subset of $\Omega$ and $x_0 \in \omega$. Then the linear heat equation \eqref{eq:HeatControl} is small-time blow-up controllable in $\{x_0\}$, i.e.~for every $T>0$, $u_0 \in L^{\infty}(\Omega)$, there exists $h \in L^{\infty}_{loc}([0,T);L^{\infty}(\omega))$ such that the solution $u$ of \eqref{eq:HeatControl} blows up at time $t=T$ and $B(u_0) = \{x_0\}$.
More precisely, for every $p>1$, $T>0$, $u_0 \in L^{\infty}(\Omega)$, there exists $h \in L^{\infty}_{loc}([0,T);L^{\infty}(\omega))$ such that for all $R>0$,
\begin{equation}
\label{eq:buprofile}
\sup\limits_{\left\{|x-a|\leq R
\sqrt{(T-t)|\log(T-t)|}\right\}}\left|(T-t)^{\frac{1}{p-1}}u(t,x)-f\left(\frac{x-a}{\sqrt{(T-t)|\log(T-t)|}}\right)\right|\xrightarrow[t \to T]{} 0,
\end{equation}
where 
$$
f(\eta)=\Big(p-1+\frac{(p-1)^2}{4p}|\eta|^2\Big)^{-\frac{1}{p-1}}, \qquad \forall \eta\in \mathbb{R}.
$$
}
\end{atheorem}

In \cite{LZ22}, the authors also prove that if $x_0 \in \Omega \setminus \overline{\omega}$, then the linear heat equation \eqref{eq:HeatControl} is not small-time blow-up controllable in $\{x_0\}$. Roughly speaking, the blow-up set of a blowing-up (controlled) solution cannot be a single point located outside the control zone.\\

Our main results of this part focus on the case of global blow-up or regional blow-up. We first have the following result of global blow-up controllability.

\begin{tm}
\label{tm:mainresultGlobalBU}
Let $\omega \subset \Omega$ be a nonempty open subset of $\Omega$. Then the linear heat equation \eqref{eq:HeatControl} is small-time blow-up controllable in $\overline{\Omega}$ i.e.,~for every $T>0$, $u_0 \in L^{\infty}(\Omega)$, there exists $h \in L^{\infty}_{loc}([0,T);L^{\infty}(\omega))$ such that the  solution $u$ of \eqref{eq:HeatControl} blows up at time $t=T$ and $B(u_0) = \overline{\Omega}$. 

More precisely, for every $p \in (1,2)$, $T>0$, $u_0 \in L^{\infty}(\Omega)$, there exist
constants $k,K,c, C>0$ depending on $T$, $\Omega$, $\omega$, $u_0$, $p$, 
and a control $h \in L_{loc}([0,T);L^{\infty}(\omega))$, defined as
\begin{equation}
	\label{eq:ControlFeedback}
	h(t) =	\left\{
			\begin{array}{ll}
				 k 1_{\omega}, & t \in (0,T/2),\\
				 	 \noalign{\vskip1mm}
				K u \log^{p}(2+|u|) 1_{\omega}, & t \in (T/2,T), 
				\end{array}
				\right.
\end{equation} 
such that
\begin{equation}
\label{eq:blowupcontrol}
c\exp\bigl[c(T-t)^{-1/(p-1)}\bigr] \le \frac{u(t, x)}{\delta(x)} \le C \exp(C(T-t)^{-1/(p-1)}),\quad t \in (T/2,T),\ x \in \Omega.
\end{equation}
\end{tm}

We secondly can show the following result of regional controllability,
with qualitatively precise localization of the regional blow-up set.

\begin{tm}
\label{tm:buregional}
Let $\omega \subset \Omega$ be a nonempty open subset of $\Omega$. 
%and assume without loss of generality that with $0\in\omega$. 
Then the linear heat equation \eqref{eq:HeatControl} is regionally small-time blow-up controllable i.e.,~for every $T>0$, $u_0 \in L^{\infty}(\Omega)$, there exists $h \in L^{\infty}_{loc}([0,T);L^{\infty}(\omega))$ such that the solution $u$ of \eqref{eq:HeatControl} blows up at time $t=T$ and the blow-up 
is regional. 

More precisely,  for every $T>0$, $u_0 \in L^{\infty}(\Omega)$, 
$r_0>0$, $x_0\in\omega$, there exist
numbers $k,\kappa,c,C>0$, $c_1>1$, $\ge 2$,
$\varepsilon\in (0,T/2)$, $r\in(0,r_0)$,
%and a function $x_0: [T/2,T)\to \Omega$ 
and a control $h \in L^{\infty}_{loc}([0,T);L^{\infty}(\omega))$, defined as
\begin{equation}
\label{eq:ControlFeedbackRegional}
h(t)=
\left\{
\begin{array}{ll}
k1_{\omega}, & t\in(0,T-\varepsilon),\\[1mm]
\kappa r^{-2}u\log^{2}(a+|u|)\,1_{B(x_0,r)}, & t\in(T-\varepsilon,T),
\end{array}
\right.
\end{equation}
such that
\begin{equation}
\label{eq:blowupcontrolRegionalprecise}
\overline{B(x_0,r)} \subset B(u_0) \subset B(x_0,c_1 r) \subset \omega,
\end{equation}
 and
$$c\exp\bigl[c(T-t)^{-1}\bigr] \le u(x,t)\le C\exp\bigl[C(T-t)^{-1}\bigr],\quad T/2<t<T,\ x\in \bar B_r.$$
%\begin{equation}
%\label{eq:blowupcontrolRegional}
%B(u_0) \subset \big\{x\ ;\, {\rm dist}(x,\omega)\le C_0K^{-1/2}\big\},\qquad K>C_0^2\, (\sup_{x \in \Omega \setminus \omega} {\rm dist}(x, \omega))^{-2}
%\end{equation}
%and 
%$$C\exp\bigl[C(T-t)^{-1}\bigr] \ge \min_{x\in \bar B_\rho(x_0(t))} u(t,x) \ge c\exp\bigl[c(T-t)^{-1}\bigr],\quad T/2<t<T.$$
Here $\kappa, c_1$ depend only on $\Omega$, $a\ge 2$ depends only on $N$, and
$r,k,c,C,\varepsilon$ depend only on $T$, $\Omega$, $\omega$,~$u_0$.
\end{tm}

The following remarks are in order:

\begin{rmk}  From a modelling perspective, results on blow-up controllability are particularly relevant for practical applications. Indeed, if
$\Omega$ represents a container where a chemical reaction can take place and \( u(t, x) \) denotes the temperature at time \( t \) and at a point \( x \in \Omega \), the dramatic increase of \( u \) at a prescribed point as in \Cref{tm:linzaag}
or in the whole \( \Omega \) as in \Cref{tm:mainresultGlobalBU}, or in a subset as in Theorem \ref{tm:buregional}  by heating in the prescribed set \( \omega \), could lead to the ignition of the chemical reaction. 
\end{rmk}

\begin{rmk}
\label{RemStrategy}
%(i) In Theorem \ref{tm:buregional}, the assumption $\omega \neq \Omega$ is not restrictive because one can assume that the control $h$ is only active on a subset $\omega_0 \subset \omega$.
%\smallskip 
(i) The proofs of Theorems~\ref{tm:mainresultGlobalBU}-\ref{tm:buregional} and \Cref{tm:linzaag} by Lin and Zaag share some similarities. Specifically, the control strategy is split into two parts: in the first time interval \( (0, T_1) \), it consists of finding a control steering the solution to a well-prepared datum, while in the second part \( (T_1, T) \), it consists of taking the control as a nonlinear closed-loop (or feedback) control \( h(t) = F(u(t)) 1_{\omega} \) for a suitable nonlinear function \( F \). 

Lin and Zaag take the control \( h(t) = F(u(t)) 1_{\omega} = |u(t)|^{p-1} u(t) 1_{\omega} \) for \( p > 1 \) to obtain the single-point blow-up controllability of \eqref{eq:HeatControl}. 
In our case, the second part of the control strategy is mainly based on our 
results on problem \eqref{eqfu} (Theorems~\ref{cor:globalblowup0BIS} and \ref{cor:globalblowup0TER}).
Namely, we take the control \( h(t) = F(u(t)) 1_{\omega} = K u(t) \log^{p}(2+|u(t)|) 1_{\omega} \) for \( p \in (1, 2) \) (respectively, \( p = 2 \)) to obtain the global (respectively, regional) blow-up controllability of \eqref{eq:HeatControl}. 
For instance, for Theorem \ref{tm:mainresultGlobalBU}, starting from the well-prepared intermediate state, the parameter $K$ is chosen such that the blow-up of \eqref{eqfu} 
with $f(x,u)= K u(t) \log^{p}(2+|u(t)|) 1_{\omega}$ happens exactly at time \( t = T - T_1=T/2 \). 
An important step in our analysis is thus to ensure the continuity of the blow-up time of this problem with respect to $K$, a property established in \Cref{tm:cbt}.
 We highlight that the well-prepared intermediate state can be any positive initial data and
our proof completely bypasses the use of the null-controllability of the heat equation.

\smallskip

(ii) Let us now present and compare different control strategies. We focus on the global blow-up control result.
\begin{enumerate}
\item \textbf{Pure feedback control.} Classically, the first part of the control strategy, i.e.~the exact-controllability to a well-prepared initial datum is mainly based on the small-time null controllability of the heat equation, recalled in \Cref{tm:nullcontrollable}. With a bit of extra work, one can also construct such a control as a closed-loop control \( h(t) = \mathcal{K}(t) u(t) \), where \( \mathcal{K}(t) \) is a linear operator in \( L^2(\Omega) \) depending on \( t \). To do this, Lin and Zaag follow the Riccati approach from \cite{Sir02}. Another recent and simpler construction could also be done using \cite{Xia24}. To sum up, with such a first control part, the control can be decomposed as follows
\begin{equation}
	\label{eq:ControlFeedbackFull}
	h(t) =	\left\{
			\begin{array}{ll}
				 \mathcal{K}(t) u(t) 1_{\omega}, & t \in (0,T/2),\\
				 	 \noalign{\vskip1mm}
				 u \log^{p}(2+|u|) 1_{\omega}. & t \in (T/2,T), 
				\end{array}
				\right.
\end{equation} 
One advantage of this control type is the fact that it is a pure feedback control with $h(t) = K(t) u(t)$ where $K$ is a nonlinear operator that does not depend on the initial data $u_0$. One drawback is that is uses the (difficult) result of \Cref{tm:nullcontrollable}.
\item \textbf{Simple feedback control depending on the data.} On the other hand, the control strategy designed in \eqref{eq:ControlFeedback} looks rather simple and natural. Choose a suitable big constant $k>0$ to steer the data to a well prepared positive state by taking $h(t) = k 1_{\omega}$, then apply the nonlinear feedback $h(t) = K u \log^{p}(2+|u|) 1_{\omega}$ with a suitable $K$ to make the solution blow-up. Despite its simple form, we would like to insist that the constants $k,\ K$ depend on $u_0$ in a non explicit way so the control takes the following form \( h(t) = K(t, u_0) u(t) \). This is not a pure feedback control. A main advantage of this control strategy is that it only use the %%the 
continuous dependence of $T^*$ with respect to $K$ and the initial data. In the same way, one could even simplify a little bit the preceding feedback by taking
\begin{equation}
	\label{eq:ControlFeedbackGlobaBis}
	h(t) =	\left\{
			\begin{array}{ll}
				 k 1_{\omega}, & t \in (0,T_1),\\
				 	 \noalign{\vskip1mm}
				 u \log^{p}(2+|u|) 1_{\omega}, & t \in (T_1,T), 
				\end{array}
				\right.
\end{equation} 
with $k = k(T,u_0)$, $T_1 = T_1(T,u_0)$ depend on $T$ and $u_0$ in non explicit way. A fully nonlinear feedback law
\begin{equation}
	\label{eq:ControlFeedbackGlobaTer}
	h(t) =	K (2+|u|) \log^{p}(2+|u|) 1_{\omega}, 
\end{equation} 
where $K$ depends also on $T$, $u_0$ in a non explicit way. See \Cref{rmk:otherpossiblecontrolstrategies} below.
\item \textbf{Open loop control.} If we do not want to pursue the objective of finding a control in a feedback form, one can rather takes an open loop control of the following form
\begin{equation}
h(t) = e^{(T-t)^{-\alpha}} 1_\omega,\quad \alpha >1. 
\end{equation}
Indeed such a control would lead to global blow-up. This mainly comes from standard heat kernel estimates, see Step 4 of the proof of Theorem~\ref{cor:globalblowup0BIS2}.
\end{enumerate}
\end{rmk}

\begin{rmk} 
(i) The results of Theorems \ref{tm:lin}, \ref{tm:linzaag}, \ref{tm:mainresultGlobalBU} and \ref{tm:buregional} 
are exactly in the opposite direction of \cite{FCZ00} and \cite{LB20}, which considered semilinear heat equations with weak nonlinearities \( |u| \log^p(1+|u|) \), \( p \in (1, 2) \), and obtained global null controllability results with the help of a localized control. In particular, the control is able to prevent the blow-up from happening.

\smallskip
(ii) The generalizations of the previous results about the internal blow-up controllability of the heat equation to coupled linear parabolic systems, in the spirit of \cite{HLL21} and considering the case of ordinary differential systems, could be interesting to consider for future research.
\end{rmk}

\medskip

\noindent \textbf{Acknowledgments.} The authors would like to warmly thank Sylvain Ervedoza
for interesting discussions during the preparation of this work.

\section{Lower estimates}
\label{seclower}

\subsection{Results for more general nonlinearities}

In the rest of this paper %%this section 
we denote
$$\beta=\frac{1}{p-1}.$$
The lower estimates in Theorems~\ref{cor:globalblowup0BIS} and \ref{cor:globalblowup0TER} 
 will be consequences of the following result for more general nonlinearities satisfying an upper growth assumption. More precisely, the following result establishes the left hand side estimate of \eqref{twosidedBU}, the part (i), the estimate \eqref{lowerBUreg}, the parts (iv) and (v) of Theorem \ref{cor:globalblowup0TER}.

\begin{thm}
\label{cor:globalblowup0BIS2}
Let $1<p\le 2$, $K>0$, and consider problem \eqref{eqfu} where 
\begin{equation}
\label{eq:hypothesisf}
0\le f(x,s)\le Ks\log^p(2+s),\quad (x,s) \in \Omega \times [0,\infty),
\end{equation}
and $u_0$ satisfies \eqref{eq:positivedata} and $T=T(u_0)<\infty$.

\begin{itemize} 
\item[(i)] Assume $p<2$.
Then $u$ undergoes global blow-up and satisfies the lower estimate:
\begin{equation}
\label{lowerBUB}
u(t,x)\ge C_1\exp\bigl[C_2(T-t)^{-\beta}\bigr]\delta(x),\quad (T-\tau)_+\le t<T,\ x\in\Omega,
\end{equation}
where $C_1=c_1(\Omega,p)K^{-\frac{N+1}{2}}>0$, $C_2=c_2(p)K^{-\beta}>0$
and $\tau>0$ depends only on $\Omega,p,K,\|u_0\|_\infty$.

\item[(ii)] Assume $p=2$.
\begin{itemize} 
\item[(ii1)] %%%
There exists {$K_0=K_0(\Omega)>0$}, such that,
if $K\le {K_0}$, then $u$ undergoes global blow-up ($B(u_0)=\overline\Omega$)
and satisfies \eqref{lowerBUB} with $p=2$.

\item[(ii2)]  {Assume $K>K_0$.}
Then $u$ undergoes at least regional blow-up. 
Namely, there exists an open ball $B\subset\subset\Omega$ such that
$\ds\lim_{t\to T} \bigl(\min_{x\in\overline B}u(t,x)\bigr)=\infty$. 
Moreover, $u$ satisfies the lower blow-up estimate:
\begin{equation}
\label{lowerBUreg2}
{u(t,x)}\ge C_1\exp\bigl[C_2(T-t)^{-1}\bigr]{\delta(x)},
\quad t_0\le t<T,\ {x\in \Omega\cap B_\rho(x_0(t))}
\end{equation}
for some function $x_0: [t_0,T)\to \Omega$,
where $t_0=\max(T/2,T-\tau)$, $C_1=c_1K^{-\frac{N+1}{2}}$, $C_2=c_2K^{-1}$, {$\rho=c_3K^{-1/2}$, 
with $c_1=c_1(\Omega)>0$, $c_3=c_3(\Omega)>0$} and $c_2>0$ a universal constant.

\item[(ii3)]
{Assume $K>K_0$,} let $\omega\subset\Omega$ be open and assume 
$f(x,\cdot)=0$ for $x\in\Omega\setminus\omega$. Then, in (ii2) the ball $B$ can be taken such that $B \subset\subset \omega$ and the function $x_0$ can be taken such that $x_0: [t_0,T)\to \omega$. Moreover if $K\le {c_4}\,diam^{-2}(\omega)$,
{where $c_4=c_4(\Omega)>0$,} then 
$B(u_0)\supset\overline\omega$ and 
$$u(t,x)\ge C_1\exp\bigl[C_2(T-t)^{-1}\bigr]\delta(x),\quad (T-\tau)_+\le t<T,\ x\in\omega,
$$
where $C_1,C_2$ are as in assertion (ii2)
and $\tau>0$ depends only on $\Omega,K,\|u_0\|_\infty$.
\end{itemize}
\end{itemize} 
 \end{thm}

 \begin{rmk}
\label{Remy0}
It follows from \eqref{lowerBUreg2} and the regularity of $\Omega$ that $u$ satisfies
 \begin{equation}
\label{lowerBUreg2Bis}
\min_{x\in \bar B_{\rho'}(y_0(t))} u(t,x)\ge c\exp\bigl[c(T-t)^{-1}\bigr], 
\quad t_0\le t<T,
\end{equation}
 for some $\rho'\in(0,\rho)$ and some function $y_0: [t_0,T)\to \Omega$, with $B_{2\rho'}(y_0(t))\subset\Omega$.
 \end{rmk}

\subsection{Preliminaries and notation}

Recalling the notation 
$$ \delta(x) = \mathrm{dist}(x, \partial\Omega),$$
we define for all $1 \leq q \leq \infty$, 
$$L_{\delta}^q= L_{\delta}^q(\Omega) = L^q(\Omega), \delta(x) dx).$$
For $1 \leq q < \infty$, $L_{\delta}^q$ is endowed with the norm
$$ \|u\|_{L_{\delta}^p} = \left(\int_{\Omega} |u(x)|^q \delta (x) dx\right)^{1/q}.$$
We note that $L^{\infty}_{\delta}(\Omega) = L^{\infty}(\Omega)$ with same norm.
Let $\Phi_1$ be the first eigenfunction of the Dirichlet Laplacian in $\Omega$,
normalized by $\int_\Omega\Phi_1\,dx=1$. Recall that, by the positivity of the first eigenfunction inside the domain $\Omega$ and Hopf's lemma, see for instance \cite[Section 1. Preliminaries]{QS19}, we have the equivalence 
\be{HopfDelta}
c\delta\le\Phi_1\le C\delta,
\ee
for some positive constants $c, C>0$ depending on the domain $\Omega$.

\smallskip

Let $p>1$, $a\ge 2$, $K>0$, we define
$$f(s)=f_K(s) = K s \log^p(a+s),\qquad s \in [0,\infty),$$
and
\be{defHODE}
H_K(s)=\int_s^\infty d\tau/f_K(\tau)=K^{-1}H_1(s)<\infty,\qquad s \in (0,\infty).
\ee
Recall that the solution of the ODE $v'= f(v)$, such that $v(t) \to +\infty$ as $t \to T$, is given by 
\be{defvODE}
v(t) = H_K^{-1}(T-t)\strut=H_1^{-1}(K(T-t)).
\ee
Moreover elementary computations show that it satisfies
\be{behaviorODE}
c_0 \exp[c_0K^{-\beta}(T-t)^{-\beta}]\le v(t)\le C_0 \exp[C_0K^{-\beta}(T-t)^{-\beta}],\quad 0\le t<T,
\ee
where $c_0,C_0>0$ depend only on $p$.

\subsection{Proof of Theorem~\ref{cor:globalblowup0BIS2}} 

\smallskip

{\bf Step 1.} {\it Lower $L^\infty$ blow-up estimate.} We claim that there exists $c_0=c_0(p)>0$ 
such that
\begin{equation}
\label{eq:lowerblowupestimate}
\|u(t)\|_\infty\ge c_0 \exp[c_0K^{-\beta}(T-t)^{-\beta}],\quad 0\le t<T.
\end{equation}

We use a comparison argument with the ODE solution $v$ in \eqref{defvODE} (see, e.g.,~the proof of \cite[Proposition 23.1]{QS19}).
From \eqref{behaviorODE}, it suffices to prove that $$\|u(t)\|_{\infty} \geq v(t).$$
Assume for contradiction that there exists $t_0 \in [0,T)$ such that $\|u(t_0)\|_{\infty} < v(t_0)$. Note in particular that $v$ is a solution to $v'(\tau) = F(v(\tau)) \geq f(x, v(\tau))$ and $\|u(t_0)\|_{\infty} < v(t_0)$. We have that for some $\varepsilon >0$, $\|u(t_0)\|_{\infty} \leq v(t_0-\varepsilon)$. By the comparison principle, we then deduce that $0 \leq u(t,x) \leq v(t-\varepsilon)$ for $(t,x) \in (t_0,T) \times \Omega$, so it is bounded in $(t_0,T) \times \Omega$, this is a contradiction. So \eqref{eq:lowerblowupestimate} holds.

\smallskip

{\bf Step 2.} {\it Eigenfunction argument.}
We claim that there exists $c>0$ depending only on $\Omega$ such that:
\begin{equation}
\label{eq:step2eigenargument}
\int_0^t \int_\Omega f(x,u(s,x))\delta(x) dxds\ge c\|u(t)\|_{L^1_\delta}-\|u_0\|_\infty,\quad 0<t<T.
\end{equation}
Set $\phi(t)=\int_\Omega u(t,x)\Phi_1(x)dx$. Testing \eqref{eqfu} with $\Phi_1$, we get
\be{eq:step2eigenargument2}
\phi'(t)+\lambda_1 \phi(t)=\int_\Omega f(x,u(t,x))\Phi_1(x)dx.
\ee
The second left hand side term is non-negative so, by integrating in time, we obtain
$$\int_0^t \int_\Omega f(x,u(s,x))\Phi_1 dxds\ge \phi(t)-\phi(0),$$
and inequality \eqref{eq:step2eigenargument} follows from \eqref{HopfDelta}.

\smallskip

{\bf Step 3.} {\it Smoothing effect and time-space lower bound on the source term.}
We claim that there exist $c_0=c_0(p)>0$, $C=C(p,\Omega)>0$ 
and $\tau=\tau(\Omega,p,K,\|u_0\|_\infty)>0$ such that
\begin{equation}
\label{eq:smoothingeffect}
\int_0^t \int_\Omega f(x,u(s,x))\delta(x) dxds
\ge CK^{-\frac{N+1}{2}}\exp[c_0K^{-\beta}(T-t)^{-\beta}],\quad (T-\tau)_+\le t<T.
\end{equation}

Let $m=(N+2)/(N+1)$ and note that $s\log^p(2+s)\le C(p)(s^m+s)$ for all $s\ge 0$.
Fix $t_1\in(0,T)$ and set
\be{defvu}
v:=\lambda e^{-\mu(t-t_1)}u,\quad\hbox{where $\lambda=K^{N+1}$ and $\mu=C(p)eK$.}
\ee
For each $t_2\in(t_1,T)$ with 
\be{condt12}
t_2-t_1\le \mu^{-1}=(C(p)eK)^{-1},
\ee
 using $\lambda K=\lambda^m$, 
we see that $v$ satisfies
$$\begin{aligned}
v_t-\Delta v
&\le \lambda Ku\log^p(2+u)-\mu \lambda e^{-\mu(t-t_2)}u 
\le C(p)\lambda^mu^m+\big(C(p)\lambda^{m-1}-\mu e^{-1}\big) \lambda u\le \tilde C(p)v^m
\end{aligned}$$
in $\Omega\times(t_1,t_2]$, with $\tilde C(p)=C(p)e^m$.
By the comparison principle, it follows that $v(t_2)\le \bar v(t_2)$, where $\bar v$ is the solution of 
\begin{equation}
	\label{eqfu0b}
		\left\{
			\begin{array}{ll}
				 \bar v_t-\Delta  \bar v = \tilde C(p) \bar v^m& \text{ in }  (t_1,t_2] \times \Omega, 
				\\
				\bar v = 0 & \text{ on } (t_1,t_2]\times \partial \Omega,
								\\
				\bar v(t_1, \cdot) = v(t_1, \cdot) & \text{ in } \Omega.
			\end{array}
		\right.
\end{equation}
We now apply to problem \eqref{eqfu0b}
 the $L^q_\delta-L^\infty$ smoothing effect for nonlinear heat equations in \cite{FSW01},
which provides the existence of 
$c_1\in(0,1)$ and $C_1>0$
depending only on $\Omega,p$, such that
\be{smooth-eff}
t_2-t_1\le c_1\bigl(\|\bar v(t_1)\|_{L^1_\delta}+1\bigr)^{-2/(N+1)}
\Longrightarrow \|\bar v(t_2)\|_\infty \le C_1(t_2-t_1)^{-\frac{N+1}{2}}\|\bar v(t_1)\|_{L^1_\delta}.
\ee
More precisely, since $m<1+2/(N+1)$, this follows from
Theorem~4, Remark~3.2(b) and formula (3.10) in \cite{FSW01}, where the quantities 
$p,q,M,K,T$ therein can be chosen as $p=m$, $q=1$, $M=\|\bar v(t_1)\|_{L^1_\delta}$, 
$K=2(M+1)$ and $T=C(m,\Omega)(M+1)^{-a}$ with $\frac{1}{a}=\frac{1}{m-1}-\frac{N+1}{2}=\frac{N+1}{2}$.
For any $t\in(0,T)$, recalling condition \eqref{condt12}, we set
$$\tau_0=\min\Bigl\{c_1\bigl(\lambda\|u(t)\|_{L^1_\delta}+1\bigr)^{-2/(N+1)},\ts\frac{T-t}{2},(C(p)eK)^{-1}\Bigr\}.$$
By \eqref{smooth-eff} with $t_1=t$ and $t_2=t+\tau_0$,
recalling the definition \eqref{defvu} of $v$ and $v(t_2)\le \bar v(t_2)$, 
it follows that
$$\begin{aligned}
\lambda e^{-\mu\tau}\|u(t+\tau_0)\|_\infty 
&\le C_1\tau_0^{-\frac{N+1}{2}} \lambda\|u(t)\|_{L^1_\delta}
\le C\bigl(\lambda\|u(t)\|_{L^1_\delta}+1+(T-t)^{-\frac{N+1}{2}}+K^{\frac{N+1}{2}}\bigr)^2
\end{aligned}$$
where, here and below, $C$ is a generic constant depending only on $\Omega,p$.
Combining this with \eqref{eq:lowerblowupestimate},
we obtain
$$\lambda\|u(t)\|_{L^1_\delta}+1+(T-t)^{-\frac{N+1}{2}}+K^{\frac{N+1}{2}}
\ge C \lambda^{1/2}\|u(t+\tau)\|^{1/2}_\infty
\ge C \lambda^{1/2} \exp[c_0K^{-\beta}(T-t)^{-\beta}],$$
hence
$$\|u(t)\|_{L^1_\delta}
\ge CK^{-\frac{N+1}{2}}\exp[c_0K^{-\beta}(T-t)^{-\beta}]-K^{-N-1}\big(1+(T-t)^{-\frac{N+1}{2}}\big)-K^{\frac{N+1}{2}}.$$
Next using \eqref{eq:step2eigenargument}, we deduce that
$$\begin{aligned}
\int_0^t &\int_\Omega f(x,u(s,x))\delta(x) dxds\\
&\ge CK^{-\frac{N+1}{2}}\exp[c_0K^{-\beta}(T-t)^{-\beta}]-\bar CK^{-N-1}\big(1+(T-t)^{-\frac{N+1}{2}}\big)-\bar CK^{\frac{N+1}{2}}
-\|u_0\|_\infty,
\end{aligned}$$
which readily implies \eqref{eq:smoothingeffect}.

\smallskip

{\bf Step 4.} {\it Heat kernel estimate and proof of Theorem~\ref{cor:globalblowup0BIS2}(i), (ii1) and {second part of} (ii3).}
By the variation of constants formula, we have:
\be{varconstu}
u(t,x)\ge \int_0^t\int_\Omega G(t-s,x,z)f(z,u(s,z)) dzdt,
\ee
where $G$ denotes the Dirichlet heat kernel.
Recall the sharp heat kernel estimate from \cite{Zha02} (see also \cite{Cho}):
$$G(t,x,z)\ge C_1t^{-N/2}\min\Bigl({\delta(x)\delta(z)\over t},1\Bigr) e^{-C_2{|x-z|^2\over t}},
\quad t>0,\ x,z\in \Omega,$$
with $C_i=C_i(\Omega)$, hence in particular
\be{lowerG}
G(t,x,z)
\ge C_3e^{-C_2{|x-z|^2\over t}}\delta(x)\delta(z),\quad t\in(0,T),\ x,z\in \Omega,
\ee
 with $C_3=C_1T^{-{N}/2}\min(T^{-1},D^{-2})$ and $D=diam(\Omega)$.
Let $t\in((T-\frac12\tau)_+,T)$ and set $\theta=(T-t)/2$, hence $t-\theta\ge (T-\tau)_+$.

 {Let} $\omega\subset\Omega$ be open, assume 
$f(x,\cdot)=0$ for $x\in\Omega\setminus\omega$ and set $d=diam(\omega)$.
Using \eqref{eq:smoothingeffect}, it follows that, for all $x\in\omega$,
\begin{align*}
\frac{u(t,x)}{\delta(x)}
&\ge C_3\int_0^t\int_\omega e^{-C_2{|x-z|^2\over t-s}}f(z,u(s,z)) \delta(z)dzds \\
&\ge C_3e^{-{2C_2d^2\over T-t}}\int_0^{t-\theta}\int_\omega f(z,u(s,z)) \delta(z)dzds \\
&\ge CC_3K^{-\frac{N+1}{2}} 
\exp\big[c_0K^{-\beta}(T-t+\theta)^{-\beta}-2C_2d^2(T-t)^{-1}\big],
\end{align*}
hence
\be{fracudelta0}
\inf_{x\in \omega}\frac{u(t,x)}{\delta(x)}\ge CC_3K^{-\frac{N+1}{2}} 
\exp\big[\bar c_0K^{-\beta}(T-t)^{-\beta}-2C_2d^2(T-t)^{-1}\big]
\ee
with $\bar c_0=(\ts\frac32)^{-\beta}c_0$.

$\bullet$ First consider the case $p<2$, hence $\beta>1$. Taking $\tau=\tau(\Omega,p,K,\|u_0\|_\infty)>0$ possibly smaller, 
\eqref{fracudelta0} with $\omega=\Omega$ yields
\be{fracudelta}
\inf_{x\in \Omega}\frac{u(t,x)}{\delta(x)}
\ge CC_3K^{-\frac{N+1}{2}} \exp\big[\ts\frac12\bar c_0K^{-\beta}(T-t)^{-\beta}\big],\quad (T-\tau)_+\le t<T,
\ee
which proves Theorem~\ref{cor:globalblowup0BIS2}(i).

$\bullet$ Now consider the case $p=2$, hence $\beta=1$. Then \eqref{fracudelta0} with $\omega=\Omega$
still guarantees \eqref{fracudelta} provided $K\le {\bar K}(\Omega):=\bar c_0(4C_2)^{-1}D^{-2}$.
This proves Theorem~\ref{cor:globalblowup0BIS2}(ii1).
Finally, the conclusion of the second part of Theorem~\ref{cor:globalblowup0BIS2}(ii3) follows similarly 
from \eqref{fracudelta0} for $K\le \bar c_0(4C_2)^{-1}d^{-2}$.

\smallskip

{\bf Step 5.} {\it Proof of Theorem~\ref{cor:globalblowup0BIS2}(ii2) {and first part of (ii3)}.}
 Since now $(T-t)^{-\beta}=(T-t)^{-1}$ is of same order as the
singularity of the heat kernel  (and $K$ is not assumed to be small), we modify the argument as follows.
 Let $c_0,C_2$ be given by \eqref{eq:smoothingeffect} and \eqref{lowerG}, respectively.
Set $\rho=\frac14({c_0\over C_2K})^{1/2}{\strut=c_3(\Omega)K^{-1/2}}$ and cover $\overline\Omega$ by finitely many balls $B_i=B_\rho(x_i)$, $x_i\in \Omega$, $i\in\{1,\dots,n_0\}$. 

Next set 
$$h(s,z)=e^{-{8C_2\rho^2\over T-s}}f(z,u(s,z))\delta(z).$$ 
 By \eqref{eq:smoothingeffect}, for all $t\in((T-\tau)_+,T)$, we have
\be{inthsz}
\int_0^t \int_\Omega h(s,z) dzds\ge ce^{-{8C_2\rho^2\over T-t}+{c_0\over K(T-t)}}
\ge ce^{{c_0\over 2K(T-t)}},\quad (T-\tau)_+\le t<T,
\ee
with $c=C(\Omega,p)K^{-\frac{N+1}{2}}$.
For $t\in(T/2,T)$, set $\hat t=2t-T\in(0,t)$, so that $t-s\ge (T-s)/2$ for all $s\in (0,\hat t)$. 
 By \eqref{varconstu}, \eqref{lowerG}, for each $i\in\{1,\dots,n_0\}$, we then have
\be{infBiu}
\begin{aligned}
\inf_{x\in {\Omega\cap\strut}B_i}\frac{u(t,x)}{\delta(x)}
&\ge C_3\inf_{x\in {\Omega\cap\strut}B_i}\int_0^t \int_{\Omega\cap B_i} e^{-{C_2|x-z|^2\over t-s}}f(z,u(s,z)) \delta(z)dzds \\
&\ge C_3 \int_0^{\hat t} \int_{\Omega\cap B_i}h(s,z)dzds.
\end{aligned}
\ee
Now, by \eqref{inthsz}, we have $\int_0^T \int_\Omega h(s,z) dzds=\infty$. 
Consequently, there exists $i_0\in\{1,\dots,n_0\}$ such that
  $\int_0^T \int_{\Omega\cap B_{i_0}} h(s,z) dzds = \infty$.
Since $\hat t\to T$ as $t\to T$, it follows from \eqref{infBiu} that
\be{infBiu2}
\inf_{x\in {\Omega\cap\strut}B_{i_0}}\frac{u(t,x)}{\delta(x)}\to\infty, \quad\hbox{as $t\to T$.}
\ee
Taking a ball $B\subset\subset B_{i_0}\cap\Omega$, this implies the first part of Theorem~\ref{cor:globalblowup0BIS2}(ii2).

On the other hand,
 for each $t\in[T/2,T)$, there exists $i=i_1(t)\in \{1,\dots,n_0\}$ such that
\be{rhon0}
\int_0^{\hat t}\int_{\Omega\cap B_{i_1(t)}} h(s,z)dzds \ge\frac{1}{n_0}\int_0^{\hat t}\int_\Omega h(s,z)dzds.
\ee
Combining \eqref{infBiu} with $i=i_1(t)$, \eqref{rhon0} and \eqref{inthsz} (with $t$ replaced by $\hat t$),
and using $T-\hat t=2(T-t)$, we obtain
$$
\inf_{x\in {\Omega\cap\strut}B_{i_1(t)}}\frac{u(t,x)}{\delta(x)}
\ge \frac{cC_3}{n_0} e^{{c_0\over 4K(T-t)}},\quad \max(\ts\frac{T}{2},T-\frac12\tau)\le t<T.
$$
This implies \eqref{lowerBUreg2} with $x_0(t):=y_{i_1(t)}$.
\smallskip

The proof of the first part of Theorem~\ref{cor:globalblowup0BIS2}(ii3) follows the same arguments
{as for \eqref{infBiu2}}, by introducing a covering of $\overline\omega$ by finitely many balls $B_i=B_\rho(x_i)$, $x_i\in \omega$, $i\in\{1,\dots,n_0\}$ and by using that $f(x,\cdot)=0$ for $x\in\Omega\setminus\omega$.
\qed

\section{Upper estimates} 
\label{secupper}

\subsection{Results for more general nonlinearities}

We consider the problem
 \begin{equation}
	\label{eqfu3}
		\left\{
			\begin{array}{ll}
				 \partial_t u - \Delta u = g(u)1_\omega& \text{ in }  (0,T) \times \Omega,  
				\\
				u = 0 & \text{ on } (0,T)\times \partial \Omega, 
				\\
				u(0, \cdot) = u_0 & \text{ in } \Omega,
			\end{array}
		\right.
\end{equation}
with more general nonlinearities $g$.
In what follows we set 
$X:=\big\{\phi\in L^\infty(\Omega),\ \phi\ge 0\big\}$ and we make the convention 
\be{conv}
\hbox{$(T-t)^{-\beta}=0$ \ \ if $0<t<T=\infty$.}
\ee
Also, we denote by $\lambda_\omega$ the first eigenvalue of the Dirichlet Laplacian on $\omega$.

We have the following uniform {\it a priori} estimates up to $t=T$, valid for both global and nonglobal solutions,
with precise dependence of the constants upon the various parameters.
As a consequence, we obtain the continuous dependence of $T$ with respect to 
initial data and parameters, as well as information on 
 the blow-up set for $p\ge 2$.
 
We note that Theorems~\ref{cor:globalblowup0BIS}, \ref{cor:globalblowup0TER} and \ref{tm:cbt} are direct consequences of 
 Theorems~\ref{cor:globalblowup0BIS2}-\ref{pr:globalblowup0BIS3-0}. More precisely, the following result establishes the right hand side estimate of \eqref{twosidedBU},  the estimate \eqref{uppernew00} of Theorem \ref{cor:globalblowup0TER} and Theorem \ref{tm:cbt}.
We pay special attention to the dependence of the constants in the estimates or assertion (i)
 since it is essential for establishing the continuity of the existence time
 in assertion~(ii).

\goodbreak

\begin{thm}
\label{pr:globalblowup0BIS3-0}
Let $u_0$ satisfy \eqref{eq:positivedata} and assume
\be{hypglobalblowup0BIS3-0}
 p>1,\  q\in(1,1+\ts\frac{2}{N+2}),\  K>0,\ \bar K\ge 1, \  a\ge 2, \ s_0\ge 0,
 \ \hbox{$\omega\subset\Omega$ is a $C^2$ smooth domain}.
 \ee
Let $g\in C^1([0,\infty))$ with $g\ge 0$ satisfy
\be{deflocalized}
g(s)\ge Ks\log^p(a+s),\quad s \in  [s_0,\infty),
\ee
and
\be{deflocalized2}
g(s)\le \bar K(1+s^q),\quad s\in  [0,\infty).
\ee

\begin{itemize} 
\item[(i)] Let $u$ denote the maximal classical solution $u$ of problem \eqref{eqfu3},
with maximal existence time $T=T(u_0)\in(0,\infty]$. Then, for all $t\in(0,T)$, we have
\be{uppernew0}
\|u(t)\|_\infty
\le C_1e^{-\lambda t}\|u_0\|_\infty
+C_2\bar K^\gamma\Big(M_1+\exp\big[C_3K^{-\beta}(T-t)^{-\beta}\big]\Big)
\ee
and
\be{uppernew0delta}
\begin{aligned}
u(t,x)
&\le \Big\{C_1\Big(\|u_0\|_\infty  t^{-\frac12} 
+\bar K\|u_0\|_\infty^q\Big)e^{-\lambda t}\\
&\qquad\qquad +C_2\bar K^\gamma \Big(M_1^q+\exp\big[C_3K^{-\beta}(T-t)^{-\beta}\big]\Big)\Big\}
\delta(x)
\end{aligned}
\ee
where $C_1=C_1(\Omega,p,q)$, $C_2=C_2(\Omega,\omega,p,q)$, $C_3=C_3(N,p,q)$,
$\lambda=\lambda(\Omega)$, $\gamma=\gamma(N,q)$ 
%$C_2=C_2(K,p,q,\Omega,\omega,s_0)$, $C_3=C_3(p,q,\Omega,\omega)$ and , ?? 
 are positive constants, and
\be{uppernew0deltaM1}
M_1=s_0^\gamma+\exp\big(C_3\lambda_\omega^{1/p}K^{-1/p}\big)
 +\exp(C_2K^{-\beta}).
 \ee
 Moreover,  we may take
\be{uppernew0deltaC2}
C_2=C(\Omega,p,q)(1+r^{-\gamma})\quad\hbox{ if $\omega=B(x_0,r)$ with $r>0$.}
\ee

  %$M_1=M_0^{\nu}+\exp(C_3K^{-\beta}\tau^{-\beta})$
 
\item[(ii)]
For $\mu\ge 0$, denote by $T^*(\mu,u_0)\in(0,\infty]$ the 
existence time of the maximal classical solution of problem \eqref{eqfu3} with $g$ replaced by $\mu g$.

\begin{itemize} 
\item[(ii1)]  The function $T^*$ is continuous from $[0,\infty)\times L^\infty(\Omega)$ to $(0,\infty]$.

\item[(ii2)]  Assume $x_0\in\Omega$, $\omega=B(x_0,r)$.
Then the function $T^*=T^*(\mu,u_0,r)$ is continuous from $[0,\infty)\times L^\infty(\Omega)\times(0,\delta(x_0))$ to $(0,\infty]$.
\end{itemize}

\item[(iii)]
Assume $T<\infty$. 
\begin{itemize} 
\item[(iii1)] Assume $p=2$. Then 
	$$B(u_0)\subset \big\{x\ ;\, {\rm dist}(x,\omega)\le C_2K^{-1/2}\big\},$$
	with $C_2=C_5(N,p,q)>0$.
In particular, if $\omega\ne \Omega$ and $K>C_0^2\, (\sup_{x \in \Omega} {\rm dist}(x, \omega))^{-2}$,
then blow-up is only regional.

\item[(iii2)] If $p>2$, then $B(u_0)\subset \bar \omega$.
\end{itemize} 
\end{itemize} 
\end{thm}

\subsection{Proofs}

In view of the proof of Theorem~\ref{pr:globalblowup0BIS3-0}, we prepare the following two lemmas, where
$\varphi_\omega$ denotes the first eigenfunction of the Dirichlet Laplacian on $\omega$ 
normalized by $\int_{\omega} \varphi_\omega = 1$ and
\be{eq:defloceigen}
\phi(t) = \int_{\omega} u(t,x) \varphi_\omega(x) dx,\qquad t \geq 0.
\ee
Recall that if $\omega=B(x_0,r)$, then
\be{lambda1r}
\lambda_\omega=c_Nr^{-2},\quad \varphi_\omega(x)=r^{-N}\varphi_0((x-x_0)/r)
\ee
where $\varphi_0$ is the first eigenfunction of the unit ball.

\begin{lm} \label{lmeigenf}
Let $p>1$, $K>0$ and $\omega\subset\Omega$ be a $C^2$ smooth domain,
and let $g\in C^1([0,\infty))$ satisfy
\eqref{deflocalized} for some $s_0\ge 0$ and $a\ge 2$.
Consider problem \eqref{eqfu3} where $u_0$ satisfies \eqref{eq:positivedata} and $T=T(u_0)\in(0,\infty]$,
and recall~\eqref{conv}.
\begin{itemize}
\item[(i)]
 Let $H_1$ be defined by \eqref{defHODE}.
%For each $K_0>0$, 
We have
%There exists $M_0=M_0(K,p,\omega,s_0)>0$ such that, for all $K\ge K_0$ and all 
\be{intvarphi1}
 \phi(t) =  \int_{\omega} u(t,x) \varphi_\omega(x) dx\le 
 M_0+H_1^{-1}\bigl(\textstyle\frac{K}{2}  (T-t)\bigr),\quad t\in [0,T),
 \ee
where 
 %\be{lambda1M1r}
$M_0=s_0+\exp\big[\big(2K^{-1}\lambda_\omega\big)^{1/p}\big]$.

\item[(ii)]
If $T<\infty$ then
\be{intvarphi1b}
 \int_{(t-1)_+}^t \int_{\omega} u(s,x)\, dxds 
\le  C_\omega\Big[M_0+CH_1^{-1}\bigl(\textstyle\frac{K}{2}  (T-t)\bigr)\Big],\quad t\in (0,T),
\ee
where the constant $C_\omega>0$ depends only on $\omega$.
Moreover,  if $\omega=B(x_0,r)$ with $r\in(0,1]$, then we may take $C_\omega=C(N)r^{-2}$.
%\be{choiceM0r}
%\tilde M_0=r^{-2}\big[s_0+\exp\big(\big(2c_nK_0^{-1}\big)^{1/p}r^{-2/p}\big)\big],\quad C(\omega)=C(n)r^{-2}.}

\end{itemize}
\end{lm}

Our second lemma guarantees the smallness of the blowup time
{when the nonlinearity $g$ is multiplied by a large coefficient
 and/or when $\omega$ is small. We especially consider the case when the RHS  in \eqref{eqfu3} is replaced by $r^{-2}g(u)1_{B(x_0,r)}$ with $r>0$ small,}
which will be important for the proof of our controllability results.

\begin{lm} \label{lmeigenf2}
Under the assumptions of Lemma~\ref{lmeigenf},
for $\mu\ge 0$, we denote by $T^*(\mu,u_0)\in(0,\infty]$, or $T^*(\mu,u_0,\omega)$, the 
existence time of the maximal classical solution of problem \eqref{eqfu3} with $g$ replaced by~$\mu g$.

 \begin{itemize}
\item[(i)]
If $\phi(0)>s_0$,
then $\lim_{\mu\to\infty} T^*(\mu,u_0)=0$. 

\item[(ii)]Assume $s_0=0$, $K\log^p a\ge C(N)>0$ sufficiently large and $u_0\ge\eta$ a.e.~in $B(x_0,r_0)$ for some $\eta,r_0>0$
and $x_0\in\Omega$.
Then
$\lim_{r\to 0}T^*(r^{-2},u_0,B(x_0,r))=0$.
\end{itemize}
\end{lm}

The proof of Lemma~\ref{lmeigenf} is based on using
the test-functions $\varphi_\omega$ and $\varphi_\omega^2$ on the equation.

\begin{proof} (i) 
By \eqref{deflocalized}, we have
$$g\ge Kh\ \ \hbox{on $[0,\infty)$, \quad where } h(s):=1_{\{s\ge s_0\}}s\log^p(2+s),\ \ s\ge 0.$$
%Fix $M_0=M_0(K_0,p,\omega,s_0)\ge s_0$ such that 
Since 
%\be{hypfgK0M0}
$K\log^p(1+M_0)>2\lambda_\omega$,
we have
	\be{hypfg}
	g(s)\ge Kh(s)\ge 2\lambda_\omega s\ \hbox{  for all $s\ge M_0$. }
	\ee
Since $\omega$ is $C^2$, we have $\varphi_\omega\in W^{2,m}(\omega)$ for all finite $m$. 
Integrating by parts and using $u\ge 0$, $\varphi_\omega=0$ and $\partial_\nu\varphi_\omega\le 0$ on $\partial\omega$, 
Jensen's inequality and the convexity of $h$,
we then obtain
\be{intvarphi1a}
\begin{aligned}
\phi'(t) 
&= \int_{\omega}u_t\varphi_\omega dx \strut \ge \int_{\omega} \varphi_\omega \Delta u dx + K\int_{\omega} h(u) \varphi_\omega dx \\
&= \int_{\omega} u\Delta\varphi_\omega dx +
 \int_{\partial\omega} \bigl[ \varphi_\omega\partial_\nu u -u\partial_\nu\varphi_\omega\bigr]d\sigma + K\int_{\omega} h(u) \varphi_\omega dx \\
& \geq - \lambda_\omega \phi(t) + K\int_{\omega} h(u) \varphi_\omega dx
\geq - \lambda_\omega \phi(t) + Kh(\phi(t)).
\end{aligned}
\ee
By \eqref{hypfg}, it follows that, for any $t\in(0,T)$,
\be{ineqphi}
\phi(t) \geq M_0\Longrightarrow \phi'(t) \geq \frac{K}{2} h(\phi(t))\Longrightarrow \phi'\geq \frac{K}{2} h(\phi) \hbox{ on $[t,T)$.}
\ee
Consequently, there exists $t_0\in[0,T]$ such that $\phi(t)\le M_0$ in $[0,t_0)$ and $\phi(t)>M_0$ in $(t_0,T)$.
In particular, \eqref{intvarphi1} is true for $t\in [0,t_0)$.
If $t_0<T$, 
by \eqref{ineqphi}, we have
\be{ineqphi2}
 \frac{K(\tau-t)}{2} \leq \int_t^\tau \frac{\phi'(s)}{h(\phi(s))} ds = \int_{\phi(t)}^{\phi(\tau)} \frac{dz}{h(z)} \leq H_1(\phi(t)),
\quad t_0\le t<\tau<T,
\ee
hence in particular $T<\infty$.
Letting $\tau\to T$ and recalling that $H_1$ is decreasing, we deduce that
$$ \int_{\omega} u(t,x) \varphi_\omega(x) dx=\phi(t)\le H_1^{-1}\bigl(\textstyle\frac{K}{2} (T-t)\bigr),\quad t_0\le t<T.$$
Therefore \eqref{intvarphi1} is also true for $t\in [t_0,T)$.

%\bleu{For further reference, we note that when $\omega=B_r(x_0)$, in view of \eqref{lambda1r}, may take}
%\be{lambda1M1r}
%M_0=s_0+\exp\Big[\big(2c_nK^{-1}\big)^{1/p}r^{-2/p}\Big].
%\ee
 
\smallskip

(ii) Set
\be{eq:defloceigen2}
 \varphi=\ts\frac12\varphi_\omega^2,\qquad \psi(t) = \displaystyle\int_{\omega} u(t,x) \varphi(x) dx,\qquad t \geq 0.
\ee
Using $\varphi=\partial_\nu \varphi=0$ on $\partial\omega$, and integrating by parts, we get
$$
\psi'(t) = \int_{\omega}u_t\varphi dx = \int_{\omega} u \Delta \varphi dx + \int_{\omega} f_K(u) \varphi dx.$$
Since $\Delta \varphi=|\nabla\varphi_\omega|^2-\lambda_\omega\varphi_\omega^2$, by integrating in time and using \eqref{intvarphi1} 
and the boundedness of $\varphi_\omega$, we obtain
$$
\int_{(t-1)_+}^t \int_{\omega} u|\nabla\varphi_\omega|^2 dxds 
=\psi(t)-\psi((t-1)_+)+\lambda_\omega \int_{(t-1)_+}^t  \int_{\omega}u \varphi_\omega^2 dx ds- \int_{(t-1)_+}^t \int_{\omega} f_K(u) \varphi dxds
$$
hence
$$
\begin{aligned}
 \int_{(t-1)_+}^t \int_{\omega} u(|\nabla\varphi_\omega|^2+\varphi_\omega) dxds 
&\le\psi(t)+ \int_{(t-1)_+}^t  \int_{\omega}u (\lambda_\omega\varphi_\omega^2+\varphi_\omega) dx ds\\
&\le (1+\lambda_\omega)(1+\|\varphi_\omega\|_\infty) \Big(M_0+%C(\omega)
H_1^{-1}\bigl(\textstyle\frac{K}{2}  (T-t)\bigr)\Big).
\end{aligned}$$
%(for a possibly larger $M_0$ with same dependence).
Observing that $\tilde c_\omega:=\inf_\omega (|\nabla\varphi_\omega|^2+\varphi_\omega)>0$
owing to Hopf's Lemma, \eqref{intvarphi1b} follows
from the last inequality, 
with $C_\omega=(1+\lambda_\omega)(1+\|\varphi_\omega\|_\infty)\tilde c_\omega^{-1}$.
 % combined with \eqref{intvarphi1}.
Moreover, if $\omega=B(x_0,r)$ with $r\in(0,1]$, using \eqref{lambda1r}, we see that we may take 
$\tilde c_\omega=c(N)r^{-N}$ and then $C_\omega=C(N)r^{-2}$.
\end{proof}

\begin{proof}[Proof of Lemma~\ref{lmeigenf2}]
(i) By \eqref{intvarphi1a}, for all $t\in(0,T)$, we have
\be{phiKphi00}
\phi(t)\ge s_0\Longrightarrow \phi'(t) \ge \big[\mu K\log^p(a+\phi(t))- \lambda_\omega\big] \phi(t),
\ee
where $a\ge 2$. Fix $\mu_1\ge 1$ such that $\mu_1 K\log^p2\ge 2\lambda_\omega$.
For any $\mu\ge \mu_1$, since $\phi(0)>s_0\ge 0$, it follows from \eqref{phiKphi00} 
that $\phi'\ge \frac12 \mu K(\log^p2)\phi(t)$ on $(0,T)$, hence
\be{phiKphi0}
\phi(t)\ge e^{\frac12 \mu K(\log^p2)t}\phi(0).
\ee
Fix $\eps>0$ and let $\mu\ge \mu_1$ be large enough so that
$\mu K\eps>2H_1(\phi(0))$ 
and $e^{\frac12 \mu K(\log^p2)\eps}\phi(0)>M_0$.
%where $M_1:=M_0(K_0,p,\omega,s_0)$ is given by Lemma~\ref{lmeigenf}(i) with $K_0:=1$. 
Then $\phi(\eps)\ge M_0$ by \eqref{phiKphi0}, and 
 it follows from  \eqref{ineqphi} and \eqref{ineqphi2} with $t_0=\eps$ that
$T(\mu,u_0)-\eps \le 2(\mu K)^{-1}H_1(\phi(\varepsilon))\le 2(\mu K)^{-1}H_1(\phi(0))\le \eps$. 
 The conclusion follows. 
 
  \smallskip
%By \eqref{phiKphi00} with $K=\tilde Kr^{-2}$
(ii) We now denote $\phi=\phi_r$ to emphasize the dependence on $r$ in \eqref{eq:defloceigen}.
Recall \eqref{lambda1r} and assume that $K\log^p a\ge 2c_N$, hence
$Kr^{-2}\log^p(2+\tilde s_0)\ge 2\lambda_\omega$.
It follows from \eqref{phiKphi00} with $s_0=0$ that
\be{phirprime}
\phi_r'(t) \ge \frac12 Kr^{-2}\phi_r(t) \log^p(a+\phi_r(t))\ \hbox{ on $(0,T^*)$.}
\ee
Also, owing to our assumption on $u_0$, for $r\in(0,r_0)$, we have $\phi_r(0)\ge \eta\int_{B_r}\varphi_{B_r}=\eta$.
By integration, we obtain
$H_1(\eta)\ge H_1(\phi_r(0))\ge \tfrac12Kr^{-2}T^*$.  The conclusion follows. 
\end{proof}

\begin{rmk}\label{rem-lmeigenf}
Under the assumptions of  Lemma~\ref{lmeigenf2} with $s_0=0$, $K\log^p a\ge C(N)>0$,
$\omega=B_r$ and $u_0 1_{B_r}\not\equiv 0$,  we have $T^*<\infty$, as a consequence of \eqref{phirprime}.
\end{rmk}

The proof of Theorem~\ref{pr:globalblowup0BIS3-0}(i) is based on Lemma~\ref{lmeigenf} and smoothing,
interpolation and semigroup arguments using the upper growth assumption \eqref{deflocalized2}.

\begin{proof}[Proof of Theorem~\ref{pr:globalblowup0BIS3-0}(i)]
Denote by $\|\cdot\|_m$ the $L^m(\Omega)$ norm for $m\in[1,\infty]$. 
In this proof, $C_1=C_1(\Omega,p,q)$ and $C_3=C_3(N,p,q)$
denote generic positive constants with this dependence.

\smallskip

{\bf Step 1.} {\it Proof of \eqref{uppernew0}.}
Recall the $L^q$-$L^\infty$ estimate for the heat semigroup:
\be{heatlambda}
\|e^{t\Delta}\phi\|_\infty\le C_0e^{-\mu t}t^{-n/2q}\|\phi\|_q,
\ee
where $C_0,\mu>0$ depend only on $\Omega$.
We may fix $\eps\in(0,1)$ and $\theta\in(0,2/(n+2))$, depending only on $q,N$, such that $q=1+\theta-\eps$. 
%%Let $q=1/\theta$. 

Let $\tau\ge 1$, $0\le t_1<t_2<T$ with $t_2-t_1\le \tau$, and set
$$U(t_1,t_2)=\sup_{s\in[t_1,t_2]} e^{\mu s}\|u(s)\|_\infty.$$
Let $t\in(t_1,t_2]$. By the variation of constants formula, we have
$$u(t,x)\le e^{(t-t_1)\Delta} u(t_1)+\bar K\int_{t_1}^t e^{(t-s)\Delta} \Big[\big(1+u^q(s)\big)1_\omega\Big]ds.$$
Using \eqref{heatlambda} and $e^{t\Delta}1_\omega\le 1$, we get
$$ \begin{aligned}
\|u(t)\|_\infty
&\le C_0e^{-\mu(t-t_1)}\|u(t_1)\|_\infty+\bar K\tau +
C_0\bar K \int_{t_1}^t  e^{-\mu(t-s)}(t-s)^{-N/2q}\|u^{1-\eps+\theta}(s)1_\omega\|_q ds\\
&\le C_0e^{-\mu(t-t_1)}\|u(t_1)\|_\infty+\bar K\tau +C_0e^{\mu(\eps-t)}\bar K 
\int_{t_1}^t (t-s)^{-\frac{N\theta}{2}}(e^{\mu s}\|u(s)\|_\infty)^{1-\eps}\|u^\theta(s)1_\omega\|_qds
\end{aligned}$$
hence
\be{emutut}
e^{\mu t}\|u(t)\|_\infty
\le C_0e^{\mu t_1}\|u(t_1)\|_\infty+\bar K\tau e^{\mu t}
+C_1\bar K U^{1-\eps}(t_1,t_2)\int_{t_1}^t(t-s)^{-N\theta/2}\|u\|_{L^1(\omega)}^\theta ds. %%C
\ee
Let $\sigma=1-\ts\frac{(N+2)\theta}{2}>0$. Since $\frac{N\theta}{2(1-\theta)}<1$, by H\"older's inequality, we have
$$\int_{t_1}^t(t-s)^{-\frac{N\theta}{2}}\|u\|_{L^1(\omega)}^\theta ds
\le \Big(\int_{t_1}^t (t-s)^{-\frac{N\theta}{2(1-\theta)}}ds\Big)^{1-\theta} \Big(\int_{t_1}^t \|u\|_{L^1(\omega)} ds\Big)^\theta \le
C_3\tau^\sigma \Big(\int_{t_1}^{t_2} \|u\|_{L^1(\omega)} ds\Big)^\theta.$$
Let $\gamma=1/\eps$. Combining this with \eqref{emutut}, taking supremum over $t\in(t_1,t_2)$ and using 
\eqref{behaviorODE}, \eqref{intvarphi1b}
and Young's inequality, we get
$$ \begin{aligned}
U(t_1,t_2)
&\le C_0e^{\mu t_1}\|u(t_1)\|_\infty+\bar K\tau e^{\mu t_2}
+C_1C_\omega^\theta\bar K \tau^\sigma U^{1-\eps}(t_1,t_2)\big(M_0+\exp\big[C_3K^{-\beta}(T-t_2)^{-\beta}\big]\big)^\theta\\
&\le C_0e^{\mu t_1}\|u(t_1)\|_\infty+\bar K\tau e^{\mu t_2}+\ts\frac12 U(t_1,t_2)
+\tilde C_2\tau^{\sigma\gamma}\bar K^\gamma
\big(M_0^\nu+\exp\big[C_3K^{-\beta}(T-t_2)^{-\beta}\big]\big),
\end{aligned}$$
where $\nu=\theta/\eps$, $\tilde C_2=C_1C_\omega^\nu$, hence
$$ \begin{aligned}
U(t_1,t_2)
&\le 2C_0 e^{\mu t_1}\|u(t_1)\|_\infty+2\bar K\tau e^{\mu t_2}+\tilde C_2\tau^{\sigma\gamma}\bar K^\gamma
\big(M_0^\nu+\exp\big[C_3K^{-\beta}(T-t_2)^{-\beta}\big]\big).
\end{aligned}$$

Next choose $\tau=\tau(\Omega)\ge 1$ large enough so that $\rho=\rho(\Omega):=2C_0e^{-\mu\tau}<\frac12$. 
Since $\gamma\ge 1$ and recalling $\bar K\ge 1$, we deduce
\be{Ut1t2}
\begin{aligned}
\|u(t_2)\|_\infty
&\le 2C_0e^{\mu(t_1-t_2)}\|u(t_1)\|_\infty+\tilde C_2\bar K^\gamma
\big(M_0^\nu+\exp\big[C_3K^{-\beta}(T-t_2)^{-\beta}\big]\big).\end{aligned}
\ee
If $T\le \tau$, then applying  \eqref{Ut1t2} with $t_1=0$ and $t_2=t$ yields 
\be{caseTle1}
\|u(t)\|_\infty\le 2C_0\|u_0\|_\infty+\tilde C_2\bar K^\gamma
\big(M_0^\nu+\exp\big[C_3K^{-\beta}(T-t)^{-\beta}\big]\big),\quad 0<t<T.
\ee

From now on we assume $T>\tau$.
Let %$D=D(K,p,q,\Omega,\omega)
$D:=\tilde C_2\big[M_0^\nu+\exp(C_3K^{-\beta}\tau^{-\beta})\big]$. We claim that
\be{inducClaim}
\|u(j\tau)\|_\infty\le \rho^j\|u_0\|_\infty+(1-\rho)^{-1}D\bar K^\gamma,\quad j=0,1,\dots,[\ts\frac{T}{\tau}-1].
\ee
Indeed this is trivial for $j=0$ and, if this is true for given integer $j$ and if $j+1\le \ts\frac{T}{\tau}-1$, 
then applying \eqref{Ut1t2} with $t_1=j\tau$ and $t_2=(j+1)\tau\le T-\tau$, we obtain
$$\|u((j+1)\tau)\|_\infty\le \rho\|u(\tau j)\|_\infty+D\bar K^\gamma
\le  \rho^{j+1}\|u_0\|_\infty+(1-\rho)^{-1}D\bar K^\gamma,$$
hence \eqref{inducClaim} follows by induction.

Set $\lambda=\lambda(\Omega)=-\tau^{-1}\log\rho>0$.
Let $t\in(0,T-\tau]$ and set $j=[\ts\frac{t}{\tau}]$. Applying \eqref{Ut1t2} with $t_1=j\tau$ and $t_2=t$
(noting that $t_2-t_1\le \tau$ and $T-t_2\ge \tau$), and next using  \eqref{inducClaim}, we obtain
$$
\|u(t)\|_\infty
\le 2C_0\|u(j\tau)\|_\infty+D\bar K^\gamma
\le 2C_0\rho^j\|u_0\|_\infty+D_1\bar K^\gamma
$$
with $D_1=(4C_0+1)D$,
hence
\be{utC0A}
\|u(t)\|_\infty
\le 2C_0\rho^{-1}e^{-\lambda t}\|u_0\|_\infty+D_1\bar K^\gamma,
\quad t\in(0,T-\tau].
\ee
Finally, if $t\in[T-\tau,T)$, applying \eqref{Ut1t2} with $t_1=T-\tau$, $t_2=t$ and \eqref{utC0A} with $t$ replaced by $T-\tau$ yields
$$ \begin{aligned}
\|u(t)\|_\infty
&\le 2C_0\|u(T-\tau)\|_\infty
+\tilde C_2\bar K^\gamma\Big(M_0^\nu+\exp\big[C_3K^{-\beta}(T-t)^{-\beta}\big]\Big)\\
&\le 2C_0\rho^{-1}e^{-\lambda t}\|u_0\|_\infty+D_1\bar K^\gamma
+\tilde C_2\bar K^\gamma\Big(M_0^\nu+\exp\big[C_3K^{-\beta}(T-t)^{-\beta}\big]\Big),
\quad t\in[T-\tau,T).
\end{aligned}$$
This along with \eqref{caseTle1},  \eqref{utC0A},
yields 
$$ \begin{aligned}
\|u(t)\|_\infty
&\le C_1e^{-\lambda t}\|u_0\|_\infty
+\tilde C_2\bar K^\gamma
\Big(M_1+\exp\big[C_3K^{-\beta}(T-t)^{-\beta}\big]\Big),\quad 0<t<T,
\end{aligned}$$
with $M_1=M_0^{\nu}+\exp(C_3K^{-\beta}\tau^{-\beta})\le M_0^{\nu}+\exp(C_1K^{-\beta})$,
which concludes the proof of \eqref{uppernew0}.

\smallskip

%$$\|u^q(t)\|_\infty\le C_1^qe^{-q\lambda t}\|u_0\|^q_\infty
%+C_5^q\bar K^{\gamma q}\Big(M_1^q+\exp\big[C_3qK^{-\beta}(T-t)^{-\beta}\big]\Big)$$

%$$\|u(t)\|_\infty\le C_1e^{-\lambda t}\|u_0\|_\infty
%+C_5\bar K^\gamma\Big(M_1+\exp\big[C_4K^{-\beta}(T-t)^{-\beta}\big]\Big)$$

{\bf Step 2.} {\it Proof of \eqref{uppernew0delta}.}
 In this proof, $C$ denotes a generic positive constant depending only on~$\Omega$.
Denote by $\lambda_1$ and $\Phi_1$ the first eigenvalue and first eigenfunction of the Dirichlet Laplacian on $\Omega$ 
normalized by $\int_{\Omega} \Phi_1 = 1$.
By the estimate $|\nabla_x G(t,x,y)|\le C\tau^{-(N+1)/2} e^{-|x-y|^2/Ct}$, 
we have 
$$G(t,x,y)\le Ct^{-(N+1)/2} e^{-|x-y|^2/Ct} \delta(x)\quad\hbox{ in $(0,\infty)\times\Omega\times\Omega$,}$$
 hence
$e^{t\Delta}1_\Omega=\int_\Omega G(t,x,y) dy\le Ct^{-1/2}\delta(x)\le Ct^{-1/2} \Phi_1(x)$.
Since $e^{t\Delta}\Phi_1=e^{-\lambda_1 t}\Phi_1$, it follows that
\be{etdelta1}
e^{t\Delta}1_\Omega\le Ct^{-1/2} e^{(t/2)\Delta}\Phi_1(x)\le Ct^{-1/2} e^{-\lambda_1 t/2}\Phi_1(x)
\le Ct^{-1/2} e^{-\lambda_1 t/2}\delta(x),\quad t>0.
\ee
By \eqref{uppernew0} and \eqref{deflocalized2}, we have
\be{futinfty}
\|g(u(t))\|_\infty
\le  \hat C_1 e^{-\lambda q t}
+\hat C_2\Big(M_1^q+\exp\big[C_3K^{-\beta}(T-t)^{-\beta}\big]\Big),
\quad 0<t<T,
\ee
with $\hat C_1=C_1\bar K\|u_0\|_\infty^q$, $\hat C_2=\bar K^{\gamma q+1}\tilde C_2^q=C_1\bar K^{\gamma q+1}C_\omega^{\nu q}$.
%$\tilde C_3=\bar K^{\gamma q+1}C_3$.
Fixing $\bar\lambda\in(0,\min(\lambda q,\lambda_1/2))$, we deduce from \eqref{etdelta1}
that
\be{etdelta20}
\begin{aligned}
\int_0^t e^{(t-s)\Delta}1_\Omega ds
&\le C \Big(\int_0^t  s^{-\frac12} e^{-\frac{\lambda_1 s}{2}}ds\Big) \delta(x)
\le C \delta(x),\quad t>0,
\end{aligned}
\ee
and
\be{etdelta2}
\begin{aligned}
\int_0^t  e^{-\lambda q s} e^{(t-s)\Delta}1_\Omega ds
&=\int_0^t  e^{-\lambda q (t-s)} e^{s\Delta}1_\Omega ds \\
&\le C \Big(\int_0^t  e^{-\bar\lambda (t-s)} s^{-\frac12} e^{-\frac{\lambda_1 s}{2}}ds\Big) \delta(x)
\le Ce^{-\bar\lambda t} \delta(x),\quad t>0.
\end{aligned}
\ee
Using the variation of constants formula
and  \eqref{etdelta1}-\eqref{etdelta2}, we obtain
$$\begin{aligned}
u(t,x)
&\le \|u_0\|_\infty e^{t\Delta} 1_\Omega +\int_0^t \|g(u(s))\|_\infty  e^{(t-s)\Delta}1_\Omega ds \\
&\le C\Big\{\|u_0\|_\infty  t^{-1/2} e^{-\lambda_1 t/2}+\hat C_1e^{-\bar\lambda t}
+\hat C_2\Big(M_1^q+\exp\big[C_3K^{-\beta}(T-t)^{-\beta}\big]\Big)\Big\}
%+\tilde C_2+\tilde C_3\exp\big[C_3K^{-\beta}(T-t)^{-\beta}\big]\Big) 
\delta(x) \\
&\le C_1\Big\{\|u_0\|_\infty  t^{-1/2} e^{-\lambda_1 t/2}+\bar K\|u_0\|_\infty^qe^{-\bar\lambda t}\\
&\qquad\qquad +C_\omega^{\nu q}\bar K^{\gamma q+1}\Big(M_1^q+\exp\big[C_3K^{-\beta}(T-t)^{-\beta}\big]\Big)\Big\}
\delta(x),
\end{aligned}$$
hence \eqref{uppernew0delta}.

 Finally, property \eqref{uppernew0deltaC2} for $\omega=B(x_0,r)$ follows immediately from 
the dependence of $C_\omega$ established in  Lemma~\ref{lmeigenf}(ii).
\end{proof}

\begin{proof}[Proof of Theorem~\ref{pr:globalblowup0BIS3-0}(ii)]
The lower semicontinuity follows from standard continuous dependence properties of the solution
with respect to parameters and initial data, as well as with respect to $r$ when $\omega=B(x_0,r)$
(the latter is a consequence of classical arguments based on Gronwall's lemma,
also using $L^m$-$L^\infty$ estimates for the heat semigroup with $m>n/2$).

Let us establish the upper semicontinuity.
We fix $\mu\ge 0$, $u_0\in X$. 
We may suppose $T_0:=T(\mu,u_0,\omega)<\infty$, 
hence in particular $\mu>0$, since otherwise there is nothing to prove.
Assume for contradiction that there exist $\eta>0$ and sequences $\mu_j\to \mu$, $X\ni u_{0,j}\to u_0$,
as well as $r_j\to r_0$ in case $\omega$ is a ball $B(x_0,r_0)$,
such that 
$T_j:=T(\mu_j,u_{0,j})\in[T_0+\eta,\infty]$.
Denote by $u$, $u_j$ the corresponding solutions.
We may assume that $\mu/2\le\mu_j\le 2 \mu$, $\|u_{0,j}\|_\infty\le M$ for some $M>0$
and $r_0/2<r_j<2r_0$.
By~\eqref{uppernew0}
(applied with $K\mu_j$, $\bar K\mu_j$ instead of $K, \bar K$, respectively),
\eqref{uppernew0deltaC2}, \eqref{uppernew0deltaM1} and \eqref{lambda1r},
there exists a constant $C$ independent of $j$ such that
$$\|u_j(t)\|_\infty\le C\big[1+\exp\big(C(T_j-t)^{-\beta}\big)\big],
\quad 0<t<T_j,\ j\ge 1.$$
In particular, since $T_j - T_0 \in [\eta,\infty]$, we obtain
$$\|u_j(t)\|_\infty\le \tilde C:=C\big[1+\exp\big(C\eta^{-\beta}\big)\big],
\quad 0<t<T_0,\ j\ge 1.$$
Passing to the limit by continuous dependence of the solution with respect to parameters and initial data
(and $r$), we deduce
that $\|u(t)\|_\infty\le \tilde C$ in $[0,T_0)$: a contradiction with $T_0=T(\mu,u_0,\omega)<\infty$.
\end{proof}

\begin{proof}[Proof of Theorem~\ref{pr:globalblowup0BIS3-0}(iii)]
Let $x_0\in\overline\Omega\setminus\overline\omega$ (assumed to be nonempty)
and let $d>0$ satisfy ${\rm dist}(x,\omega)\ge d$.
Let $x\in B_{d/2}(x_0)\cap \Omega$. 
Since ${\rm dist}(x,\omega)\ge d/2$, 
by the Gaussian upper bound for the heat kernel, we have
$$G(t-s,x,y)\le (t-s)^{-\frac{N}{2}}e^{-{|x-y|^2\over 4(t-s)}}\le (t-s)^{-\frac{N}{2}}e^{-d^2\over 16(t-s)}\le Cd^{-N}e^{-d^2\over 32(T-s)},\quad 0<s<t<T,\ y\in \omega,$$
with $C=C(N)>0$.
Assume $K>0$ if $p>2$ (hence $\beta<1$), or $K\ge 32\,C_4d^{-2}$ if $p=2$ (hence $\beta=1$).
Using the variation of constants formula and \eqref{futinfty}, for all $t\in(0,T)$, it follows that
$$\begin{aligned}
u(t,x)
&\le \|u_0\|_\infty +\int_0^t \int_\omega G(t-s,x,y)\|f(u(s,\cdot))\|_\infty  dy ds\\
&\le \|u_0\|_\infty+Cd^{-N}|\omega| T \sup_{s\in(0,T)}
\Bigl((\hat C_1+\hat C_2)e^{-d^2\over 2(T-s)}+\tilde C_3e^{{C_4\over K(T-s)^\beta}-{d^2\over 32(T-s)}}\Bigr) =:M<\infty,
\end{aligned}$$
hence $x_0\not\in B(u_0)$.
We conclude that 
$B(u_0)\subset \big\{x\ ;\, {\rm dist}(x,\omega)\le (32\,C_4/K)^{1/2}\big\}$ if $p=2$ and that $B(u_0)\subset \bar \omega$
if $p>2$.
This completes the proof.
\end{proof}

\section{Proofs of blow-up controllability results}
\label{sec:cbt}

In this part we prove Theorems~\ref{tm:mainresultGlobalBU} and \ref{tm:buregional}
as a consequence of the results in Section~\ref{secBU}.
We stress that these proofs completely bypass the use of the null-controllability of the heat equation.
%(unlike that of \Cref{tm:mainresultregionalBU}; cf.~\Cref{lm:exactcontrol} below).

We consider the initial boundary value problem
\begin{equation}
	\label{eq:HeatControl2}
		\left\{
			\begin{array}{ll}
				 \partial_t u - \Delta u = h(t,u(t,\cdot)) 1_{\omega} &  (0,+\infty)\times\Omega,
				\\
				u = 0 & (0,+\infty)\times\partial\Omega,
				\\
				u(0, \cdot) = u_0 &\Omega,
			\end{array}
		\right.
\end{equation}
where $h$ is given by \eqref{eq:ControlFeedback} (resp., \eqref{eq:ControlFeedbackRegional})
and denote by $T^*=T^*(k,K)$ the existence time of its maximal strong solution
(resp., $T^*=T^*(k,K,r)$).
Note that $T^*>T/2$ since the problem is (inhomogeneous)  linear for $t\le T/2$. 

\begin{proof} [Proof of Theorems~\ref{tm:mainresultGlobalBU} and \ref{tm:buregional}]
{\bf Step 1.} {\it Linear part of the control. On $[0,T/2]$ (resp.,~$[0,T-\eps]$), we will have $u=U:=w+kz$,} 
where $w=e^{t\Delta} u_0$ and $z$ is the solution of the linear inhomogeneous problem 
$z_t-\Delta z=1_\omega$ on $\Omega$ with $z(0)=0$ and $0$ Dirichlet conditions.
We first claim that there exists $k=k(T,\Omega,\omega, u_0)>0$ such that
\be{lowerdelta}
U(t,x)\ge \delta(x),\quad t\in[T/2,T].
\ee
Indeed, for $t\in[T/2,T]$, we have
$z(t,x)\ge c_1\delta(x)$ by the Hopf lemma, and 
$|w(t,x)|\le c_2\delta(x)$
 by $C^1$ regularity, with $c_i=c_i(T,\Omega,\omega,u_0)>0$. 
Consequently,
$y(t,x)\ge (kc_1-c_2)\delta(x)>0$ and it suffices to take $k=(1+c_2)/c_1$.

\smallskip

{\bf Step 2.} {\it Nonlinear part of the control for \Cref{tm:mainresultGlobalBU}.}
Fix a nonempty open subset $\omega'$ of class $C^2$ contained in $\omega$, $v_0=u(T/2) \not\equiv 0$ in $\omega'$, owing to \eqref{lowerdelta}, and denote by $T_*(v_0,K)$ the maximal existence time of the solution $v$ of 
problem \eqref{eqfu} with $f$ given by \eqref{eq:modelf} with $p<2$ and $u_0=v_0$.
We have 
$$T^* = T/2 + T_*(v_0,K).$$
Since $T_*(v_0,0)=\infty$, it follows from 
(the l.s.c.~part of) \Cref{tm:cbt} that $T_*(v_0,K_1)>T/2$ for $K_1>0$ small.
On the other hand, by Lemma~\ref{lmeigenf2}(ii), we have $T_*(v_0,K_2)<T/2$ for $K_2>K_1$ large.
By continuity, applying \Cref{tm:cbt} again, we deduce the existence of 
$K\in (K_1,K_2)$ such that $T_*(v_0,K)=T/2$, hence $T^*(k,K)=T$.
In view of \eqref{twosidedBU}, this completes the proof of \Cref{tm:mainresultGlobalBU}.

\smallskip

{\bf Step 3.} {\it Nonlinear part of the control for \Cref{tm:buregional}.}
With the notations of Corollary~\ref{cor:reg}, we select 
$\kappa=\tfrac14 c_1(\Omega)$ and recall $\tilde\kappa = 1 + C_0 \kappa^{-1/2}$.
{We then take $a\ge 2$ such that $\kappa\log^p a\ge C(N)$ (cf.~Lemma~\ref{lmeigenf2}(ii)),}
and next pick $r_0\in(0,\delta(x_0))$ such that $\bar B(x_0,\tilde\kappa r_0)\subset \omega$.
For $r\in (0,r_0]$, denote by $\tilde T_*(v_0,r)$ the maximal existence time of the solution of the
problem 
\begin{equation}
	\label{eqfu2}
		\left\{
			\begin{array}{ll}
				 \partial_t v - \Delta v = \kappa r^{-2} v\log^2(a+v)1_{B_r}(x),
& \text{ in }  (0,T) \times \Omega,  
				\\
				v = 0 & \text{ on } (0,T)\times \partial \Omega, 
				\\
				v(0, \cdot) = v_0 & \text{ in } \Omega.
			\end{array}
		\right.
\end{equation}
We have 
\be{decompTstar}
T^*(k,\kappa r^{-2},r) = T-\eps+ \tilde T_*(u(T-\eps),r).
\ee
Since $U$, defined in Step 1, satisfies $\sup_{t\in(0,T)}\|U(t)\|_\infty<\infty$, 
and owing to the local well-posedness of \eqref{eqfu2},
we have $\tau:=\inf_{t\in (T/2,T)} \tilde T^*(U(t),r_0)>0$.
Choose $\eps=\tfrac12\min(T,\tau)$, hence $\tilde T_*(u(T-\eps),r_0)>\eps$.
%hence $T^*(k,\kappa r_0^{-2},r_0)>T$.
Also, by \eqref{lowerdelta}, we have $\eta:=\inf_{B_{r_0}} U(T-\eps)>0$.
It follows from Lemma~\ref{lmeigenf2}(ii) that $\tilde T_*(u(T-\eps),r)<\eps$ for $r>0$ small.
By the continuity property in Theorem~\ref{tm:cbt}(ii), we deduce the existence of 
$r\in (0,r_0)$ such that $T_*(v_0,K)=\eps$, hence $T^*(k,K)=T$ by \eqref{decompTstar}.
In view of Corollary~\ref{cor:reg}, this completes the proof of \Cref{tm:buregional}.
\end{proof}

\begin{rmk}
The first part of the above proof, leading to \eqref{lowerdelta}, demonstrates by a very simple argument the so-called global non-negative controllability of the heat equation. This notion was introduced by the first author in \cite{LB20}, and was particulary relevant for the problem of global of null-controllability of weakly superlinear heat equations in the semi-dissipative case. Similar arguments already appeared in \cite{BCMR96}, see in particular the proof of Lemma 7. 
\end{rmk}

\begin{rmk}
\label{rmk:otherpossiblecontrolstrategies}
Let us justify the other possible control strategies.

First, we take the control as
\begin{equation}
	\label{eq:ControlFeedbackGlobaBisRmk}
	h(t) =	\left\{
			\begin{array}{ll}
				 k 1_{\omega}, & t \in (0,T_1),\\
				 	 \noalign{\vskip1mm}
				 u \log^{p}(2+|u|) 1_{\omega}, & t \in (T_1,T), 
				\end{array}
				\right.
\end{equation} 
for suitable $k$ and $T_1$. Let $z_k(t)$ be the solution at time $t$ to the linear inhomogeneous problem $\partial_t z_k - \Delta z_k = k 1_{\omega}$, starting from $z(0) = u_0$. Then as before, there exists $k_0>0$ such that for every $t \in [T/2,T]$, $z_{k_0}(t) \ge c \delta(x)$. Then there exists $\varepsilon \in (0,T/2)$ such that for every $t \in [T/2,T]$, $T^*(z_{k_0}(t)) > \varepsilon$ where $T^*$ is the 
existence time %blow-up 
associated to the nonlinearity  $ u \log^{p}(2+|u|) 1_{\omega}$. Let $T_1 = T-\varepsilon \in (T/2,T)$. So from now $k_0>0$ and $T_1$ are fixed.  By definition $T^*(z_{k_0}(T_1))> \varepsilon$. Moreover, it is easy to establish that $T^*(z_{k}(T_1)) \to 0$ as $k \to +\infty$. By continuity of the blow-up time with respect to the initial data, there exists $k_1 \in (k_0, +\infty)$ such that $T^*(z_{k}(T_1)) = \varepsilon$. Then the control \eqref{eq:ControlFeedbackGlobaBisRmk} leads to the global blow-up at time $T$.

Secondly, we take the control as
\begin{equation}
	\label{eq:ControlFeedbackGlobaTerRmk}
	h(t) =	K (2+|u|) \log^{p}(2+|u|) 1_{\omega}, 
\end{equation}
Using the notations of \Cref{lmeigenf2}, the conclusion follows from the continuity of the existence time with respect to $K$
and the fact that $T^*(0,u_0) = +\infty$ and $T^*(K, u_0) \to 0$ as $K \to +\infty$. 
To verify the latter property we first observe that
for given $\eta>0$ there exists $K_\eta>0$ such that $u(\eta,x) \ge \delta(x)$ for all $K\ge K_\eta$
(this follows from the argument in Step 1 of the proof of Theorem~\ref{tm:mainresultGlobalBU}
along with $h(t)\ge c(p)K1_{\omega}$).
We may then apply \Cref{lmeigenf2}(i) with $s_0=0$.
\end{rmk}

\begingroup
\small

\bibliographystyle{plain}

\bibliography{heat}

\endgroup

\end{document}